\documentclass[11pt,article]{article}

\usepackage{amsmath}
\usepackage{amsthm}
\usepackage{paralist}
\usepackage{epsfig} 
\usepackage{graphicx}
\usepackage{caption}
\usepackage{subcaption}
\usepackage{graphics} 
\usepackage{epsfig} 
\usepackage{multirow}
\usepackage{array}
\usepackage{hhline}
\usepackage[font=scriptsize,labelfont=sf]{caption}
\usepackage[margin=2cm]{caption}
\usepackage{epstopdf}
\input{amssym.tex}

\newtheorem{theorem}{Theorem}[section]
\newtheorem{corollary}{Corollary}

\newtheorem{proposition}{Proposition}[section]

 \numberwithin{equation}{section}
\newtheorem{remark}{Remark}[section]

\newcommand{\keywords}

\def\bc{\begin{center}}       \def\ec{\end{center}}
\def\ba{\begin{array}}        \def\ea{\end{array}}
\def\be{\begin{equation}}     \def\ee{\end{equation}}
\def\bea{\begin{eqnarray}}    \def\eea{\end{eqnarray}}
\def\beaa{\begin{eqnarray*}}  \def\eeaa{\end{eqnarray*}}

\def\mathbb{\Bbb}

\begin{document}

\title{\bf Stationary patterns and their selection mechanism of Urban crime models with heterogeneous near--repeat victimization effect}
\author{Yu Gu\thanks{{\tt guyu@2011.swufe.edu.cn}}, Qi Wang \thanks{{\tt qwang@swufe.edu.cn}, corresponding author.}, Guangzeng Yi\thanks{{\tt guangzengyi@2011.swufe.edu.cn}}\\
Department of Mathematics\\
Southwestern University of Finance and Economics\\
555 Liutai Ave, Wenjiang, Chengdu, Sichuan 611130, China
}

\date{}
\maketitle

\abstract
In this paper, we study two PDEs that generalize the urban crime model proposed by Short \emph{et al}. [Math. Models Methods Appl. Sci., 18 (2008), pp. 1249--1267].  Our modifications are made under assumption of the spatial heterogeneity of both the near--repeat victimization effect and the dispersal strategy of criminal agents.  We investigate pattern formations in the reaction--advection--diffusion systems with nonlinear diffusion over multi--dimensional bounded domains subject to homogeneous Neumann boundary conditions.  It is shown that the positive homogeneous steady state loses its stability as the intrinsic near--repeat victimization rate $\epsilon$ decreases and spatially nonconstant solutions emerge through bifurcation.  Moreover, we find the wavemode selection mechanism through rigorous stability analysis of these nontrivial patterns, which shows that the only stable pattern must have wavenumber that maximizes the bifurcation value.  Based on this wavemode selection mechanism, we will be able to precisely predict the formation of stable aggregates of the house attractiveness and criminal population density, at least when the diffusion rate $\epsilon$ is around the principal bifurcation value.  Our theoretical results also suggest that large domains support more stable aggregates than small domains.  Finally, we perform extensive numerical simulations over 1D intervals and 2D squares to illustrate and verify our theoretical findings.  Our numerics also include some interesting phenomena such as the merging of two interior spikes and the emerging of new spikes, etc.  These nontrivial solutions can model the well observed aggregation phenomenon in urban criminal activities.



\section{Introduction}
One of the most noticeable and most often observed phenomena in urban criminal activities is the clustering of crime data.  For example, as reported in \cite{BW,B}, certain neighbourhoods have a higher propensity to crimes than others though crimes may occur everywhere in the community.  Many works in literature are devoted to the understanding of the spatially inhomogeneous distributions of the criminal activities.  Great efforts are made to investigate the impact of social forces on such empirical observations and theories proposed for this purpose include social disorganization \cite{KW,PK,SG}, subculture \cite{F,WF} and conflict theories \cite{C,CBH} etc.

It is widely believed that human behaviors are too complex to be explained or predicted individually through mathematical models, however there are regularities of human behaviors that can be modeled and understood mathematically at group level.  In the pioneering work of \cite{SDP}, Short \emph{et al}. proposed mathematical frameworks based on a 2D lattice system and its continuum counterpart to model and simulate the urban residential burglars.  The modeling there incorporates an important empirical observation Broken windows effect \cite{SR,WK}, i.e., a building with a few broken windows which remain unrepaired attracts vandals to break a few more windows and eventually to break the whole building.

In its dimensionless form, the continuum model in \cite{SDP} is a strongly coupled reaction--advection--diffusion system that reads as follows
\begin{equation}\label{11}
\left\{
\begin{array}{ll}
\frac{\partial A}{\partial t}=\eta \Delta A-A+A^0+\rho A,\\
\frac{\partial\rho}{\partial t}=\nabla \cdot (\nabla \rho-2\rho \nabla\log A) -\rho A+\bar B,
\end{array}
\right.
\end{equation}
where $(A,\rho)=(A(x,t),\rho(x,t))$ and they denote house attractiveness and criminal population density at space--time location $(x,t)$ respectively, $\eta$ and $\bar B $ are positive constants, and $A^0$ is a nonnegative function of $x$.  (\ref{11}) is similar as the Keller--Segel models of chemotaxis in which bacteria move along the gradient of the chemical concentration.  See the survey papers \cite{Ho,HP} and \cite{KS,KS1,KS2,SWW} for works on the chemotaxis models.

It is demonstrated in \cite{SDP} that nonlinear system (\ref{11}) and its discrete counterpart can develop very complex spatial--temporal patterns.  More interestingly, the pattern formations can qualitatively capture the dynamics of residential burglar aggregates, or the so--called the crime hotspots, presented in the form of spiky steady states.  Crime hotspots describe the clustering of crime data in urban residential burglaries, i.e., there are neighbourhoods with higher crime rates surrounded by neighbourhoods with lower crime rates.  Moreover, a linear stability has been carried out to determine the parameter values that will lead to the creation of stable hotspots.

In this paper, we investigate the formation of spatially inhomogeneous patterns for following systems with heterogeneous diffusion rate $\eta(A)$ and perception function $f(A)$
\begin{equation}\label{12}
\left\{
\begin{array}{ll}
\frac{\partial A}{\partial t}=\epsilon \Delta(\eta(A)( A-A^0))-A+A^0+\rho A, &x\in \Omega,t>0,\\
\frac{\partial\rho}{\partial t}=\nabla\cdot(\nabla\rho-2\rho\nabla \log f(A))-\lambda_0\rho A+\lambda_0\bar B,  &x\in\Omega,t>0,
\end{array}
\right.
\end{equation}
and
\begin{equation}\label{13}
\left\{
\begin{array}{ll}
\frac{\partial A}{\partial t}=\epsilon\nabla\cdot\left(\eta^2(A)~\nabla\frac{A-A^0}{\eta(A)}\right)-A+A^0+\rho A, &x\in \Omega,t>0,\\
\frac{\partial\rho}{\partial t}=\nabla\cdot(\nabla\rho-2\rho\nabla \log f(A))-\lambda_0\rho A+\lambda_0\bar B, &x\in  \Omega,t>0,
\end{array}
\right.
\end{equation}
subject to homogeneous Neumann boundary conditions
\begin{equation}\label{14}
\partial_\textbf{n} A=\partial_\textbf{n} \rho=0, x\in \partial \Omega, t>0,
\end{equation}
where $\Omega\subset \mathbb R^N$, $N\geq1$, is a bounded domain with piecewise smooth boundary $\partial \Omega$; $\textbf{n}$ is the unit outer normal of the boundary.  $\eta (A)$ is a continuously differentiable function satisfying $\eta(A)>0,\eta'(A)>0$ and $\eta(A)\geq \eta'(A)A$ for all $A>0$.  $f(A)>0$ is also assumed to a continuously differentiable function.

The attractiveness $A(x,t)$ consists of two components $A(x,t)\equiv A^0+B(x,t)$, where $A^0$ represents the static part and $B$ measures the dynamic part.   Diffusion rate $\eta(A)$ is used to interpret the strength of the empirically observed repeat/near--repeat victimization effect \cite{ACCGT,JB,JBH}, which states that a house and its neighbourhoods become more likely to be burglarized soon after a burglary.  $f(A)$ is a function that measures the perception of house attractiveness hence the advection term describes the directed dispersal flux of the criminal agents to the attractive sites over the community with respect to the perception.  $\epsilon$ is a positive constant that measures the maximum strength of the near--repeat victimization effect and we shall it the intrinsic near--repeat victimization rate.  The constants $\lambda_0$, $A^0$ and $\bar B$ are assumed positive, where $\lambda_0$ measures the probability that a burglar going back home after committing a crime, $A^0$ and $\bar B$ represent the static value and spatial average value of the attractiveness of each house respectively.

The work in \cite{SDP} has triggered great interest of many scholars in the theoretical analysis and numerical studies of urban crime activities and other similar sociological phenomena.  To model the empirically well--observed hotspots from the continuum models, one way is to construct stationary solutions that have concentration structures.  In \cite{SBB}, Short \emph{et al}. performed weak nonlinear analysis based on perturbation arguments to establish the emergence and suppression of stationary hotspot patterns of (\ref{11}).  Cantrell \emph{et al.} \cite{CCM} applied the local bifurcation theory of Crandall--Rabinowitz \cite{CR} and its developed version \cite{SW} to investigate the existence and stability of nonconstant positive steady states of (\ref{11}).  By Leray--Schauder degree argument, Gacia--Huidobro \emph{et al}. proved the existence of nonconstant positive steady states to 1D burglar model in \cite{GMM2} and to its general form in \cite{GMM}.  In \cite{KWW}, Kolokolnikov \emph{et al.} constructed steady state with hotspots and investigated their stabilities when the region is a one--dimensional or two--dimensional domain.  Tse and Ward \cite{TW} applied a combination of asymptotic analysis and numerical path--following methods to construct hotspot steady states and analyzed their bifurcation properties as the diffusivity of criminals varies.  Localized patterns have been studied in \cite{LO}.  For the time--dependent system (\ref{11}), Rodri\'guez and Bertozzi \cite{RB} established the existence and uniqueness of local solutions; Rodri\'guez \cite{Ro} proved the global existence of fully--parabolic system (\ref{11}) and that of its counterpart with the logarithmic sensitivity $\nabla \log A$ replaced by $\nabla A$, provided that the initial criminal population is small.  Berestycki \emph{et al}. \cite{BRR} studied a similar system and obtained its traveling wave solutions that connect zones with no criminal activity and zones with crime hotspots.

It is worthwhile to point out that the agent--based model (\ref{11}) in \cite{SDP} has been extended in several aspects by various authors.  For example, effects of police actions on the spatial distribution of criminal activities are taken into account in \cite{JBC,Ri,SBB,ZSB}, etc.  Chaturapruek \emph{et al}. \cite{CBYKM} proposed a crime model with the criminals following a biased L\'evy flight.  It is quite difficult for us to list all the works and we refer to \cite{BLSW,MPS,NW,PDH,RR,SMBT,WMBB} and the references therein for the development of models with the same or similar sociological backgrounds.  In this paper, we investigate the formation of nonconstant positive steady states to models (\ref{12}) and (\ref{13}) subject to (\ref{14}).  We want to remark that the goal of this paper is not to present mathematical models to predict or replicate the behaviors of a single human individuals, but to present a systematic treatment of reaction--advection--diffusion systems modeling urban criminal dynamics, in particular the formation of spatially nontrivial patterns.  In particular, all of our theoretical results hold for (\ref{11}) by taking $\eta(A)\equiv \eta$, $\lambda_0=1$ and $f(A)=A$ in (\ref{12}) or (\ref{13}).  On the other hand, we also want to point out that, the nonlinearity in diffusion $\eta(A)$ and perception function $f(A)$, in particular when degeneracy is allowed, can make these models even richer in spatial--temporal dynamics.  For example, there are many works on chemotaxis models including but not limited to formation of singularity, compacted supported steady states, \emph{etc}.  We would like to expect more works on the effect of such nonlinearity on the formation and evolution of crime hotspots.

Our paper is organized as follows.  In Section \ref{section2}, we derive the continuum models (\ref{12}) and (\ref{13}) following the microscopic--macroscopic approach in \cite{SDP}.  In Section \ref{section3}, we carry out the linearized stability analysis of the homogeneous steady state $(\bar A,\bar \rho)$ and show that small diffusion rate $\epsilon$ tends to destabilize the constant equilibrium--see Proposition \ref{proposition31}.  Section \ref{section4} and \ref{section5} are devoted to studying the nonconstant positive steady states of (\ref{12}) and (\ref{13}).  By Crandall--Rabinowitz bifurcation theory, we rigorously establish the existence and stability of nonconstant steady states--see Theorem \ref{theorem41}, Theorem \ref{theorem42} and Theorem \ref{theorem51}.  We then numerically solve the crime models (\ref{12}) and (\ref{13}) in Section \ref{section6} to illustrate the formation of clustering criminal data in the form of hotspots, hotstripes, etc.  Finally we include our remarks and propose some future problems in Section \ref{section7}.

\section{Derivation of the models}\label{section2}
To understand the relevance of the mathematics presented in \cite{SDP} and the current work to the urban criminal activities, it is helpful to demonstrate how models (\ref{12}) and (\ref{13}) are derived.  We shall extend the urban crime model in \cite{SDP} to nonlinear diffusion systems (\ref{12}) and (\ref{13}) by considering spatially heterogeneous near--repeat victimization and dispersal strategies of the criminal agents.  Moreover, we take into account the nonlinear human perception of a physical stimulus.

\subsection{Derivation of the house attractiveness equation}
Consider a 2D lattice with constant spacing $l$.  Each house is located at site $s=(i,j)$, $i,j\in \mathbb N$, with four neighbouring sites $s'\in N(s)=\{(il,(j\pm1) l),((i\pm1)l,jl)\}$.  Denote attractiveness of the target house $s$ at time $t$ by $A_s(t)$ and the number of criminals by $n_s(t)$ respectively.  We want to mention that the attractiveness $A_s(t)$ is defined in a comprehensive way as in \cite{SDP} and it consists of two part $A_s(t)=A_0+B_s(t)$, $A_0$ is static and $B_s$ is dynamic.

According to the near--repeat victimization effect \cite{SDP}, the dynamic attractiveness $B_{s'}(t)$ to the neighbouring site of $s$ increases each time after it is burglarized, and the increase in $B_{s'}(t)$ is contributed by the loss of the attractiveness $B_s(t)$ at site $s$, i.e., part of $B_s(t)$ is transmitted to its neighbouring sites each time after it is burglarized.  This self--exciting phenomenon represents the strong tie of criminal activities to the attractiveness of their environments, as well as the feedback mechanism that local attractiveness is increased by criminal activities.  This mechanism is quite similar as the flowing of heat from region of high temperature to regions of low temperature.  Here we denote the near--repeat victimization effect by $\eta>0$.  Therefore the near--repeat victimization effect is strong if $\eta$ is large and it is weak if $\eta$ is small.   On the other hand, it is very likely that $\eta$ may vary from house to house.  In this paper, we assume that communities with different attractiveness values have different sensitivities to their local criminal activities, therefore $\eta$ at site $s$ takes the form $\eta_s(t)=\eta(A_s(t),t)$.  Moreover, as the transition probability depends on dynamic attractiveness $\eta_s$ of the current house (departure point) or $\eta_{s'}$ of the target house (arrival point), we divide our discussions into the following two cases.

\subsubsection{Near-repeat victimization effect $\eta$ dependent on attractiveness of departure point}
If the near--repeat victimization effect at $s$ depends on attractiveness of the departure point $s$, then the transition probability of house attractiveness takes the form
\begin{equation}\label{21}
P(s\rightarrow s')=\frac{\epsilon\eta_{s}(t)}{z},
\end{equation}
where $\epsilon$ is a positive constant that measures the maximum value of near--repeat victimization effect, $z$ is the number of neighbouring sites of $s$ ($z=4$ in 2D) and the notation $s\rightarrow s'$ denotes the shifting of dynamic attractiveness from $s$ to $s'$ due to the near--repeat victimization effect.

We assume that burglars occur at site $s$ during $(t,t+\delta t)$ following a standard Poisson process with probability $P_s(t)=1-e^{-A_s(t)\delta t}$, where $A_s(t)\delta t$ is the average number of burglars during $(t,t+\delta t)$.  Denote the total population and expected population of criminal agents at $s$ by $n_s(t)$ and $E_s(t)$ respectively.  Since the attractiveness decays to its baseline value at a rate $\omega$ if no criminal activity occurs afterwards, the attractiveness at $s$ satisfies the difference equation
\begin{equation}\label{22}
\begin{split}
B_s(t+\delta t)&=\Big(B_s(t)\big(1-\epsilon\eta_s(t)\big)+\sum\limits_{s'\sim s}P(s'\rightarrow s)B_{s'}(t)\Big)(1-\omega\delta t)+\theta E_s(t)\\
&=\Big(B_s(t)(1-\epsilon\eta_s(t))+\sum\limits_{s'\sim s}\frac{\epsilon\eta_{s'}(t)}{z}B_{s'}(t)\Big)(1-\omega\delta t)+\theta n_s(t)P_s(t)\\
&=\Big(B_s(t)+\frac{l^2\epsilon}{z}\Delta_d(\eta_s(t)B_s(t))\Big)(1-\omega\delta t)+\theta n_s(t)P_s(t),
\end{split}
\end{equation}
where the notation $s'\sim s$ indicates all the neighboring sites of $s$, $\theta$ is the increase of attractiveness due to one burglary event and $\Delta_d$ is the discrete spatial Laplacian
\begin{equation}\label{23}
\Delta_d B_s(t)=\left(\sum\limits_{s'\sim s}B_{s'}(t)-zB_s(t)\right)\Bigg/l^2.
\end{equation}
To derive the continuum PDE, we subtract $B_s(t)$ from (\ref{22}) and then divide it by $\delta t$.  Denote the criminal population density by $\rho_s(t)=\frac{n_s(t)}{l^2}$.  After sending both $\delta t$ and $l$ to zero and applying the limits as in \cite{SDP}
\[\lim_{\delta t\rightarrow 0^+}\frac{l^2}{\delta t}=D>0, ~\lim_{\delta t \rightarrow 0^+} \theta \delta t=\kappa>0,\]
we collect the following equation of the dynamic attractiveness
\begin{equation}\label{24}
\frac{\partial B}{\partial t}=\frac{\epsilon D}{z}\Delta(\eta B)-\omega B+\kappa D\rho A.
\end{equation}
where $\Delta=\sum_{i=1}^N \frac{\partial^2 }{\partial x_i^2}$ is the Laplacian in $\mathbb R^N$.  See \cite{SDP} for justifications on the limits.

\subsubsection{Near-repeat victimization effect $\eta$ dependent on attractiveness of arrival point}
If the near--repeat victimization effect $\eta$ at $s$ depends on its arrival-point $s'$, we can write the transition probability as
\begin{equation}\label{25}
P(s\rightarrow s')=\frac{\epsilon\eta_{s'}(t)}{z},
\end{equation}
where $\epsilon$ is defined to the same as for (\ref{21}).  Then the discrete equation of attractiveness leads us to
\begin{equation}\label{26}
\begin{split}
B_s(t+\delta t)&=\Big(B_s(t)\big(1-\sum _{s'\sim s}P(s\rightarrow s')\big)+\sum\limits_{s'\sim s}P(s'\rightarrow s)B_{s'}(t)\Big)(1-\omega\delta t)+\theta E_s(t)\\
&=\Big(B_s(t)\big(1-\sum\limits_{s'\sim s}\frac{\epsilon\eta_{s'}(t)}{z}\big)+\sum\limits_{s'\sim s}\frac{\epsilon\eta_s(t)}{z}B_{s'}(t)\Big)(1-\omega\delta t)+\theta n_s(t)P_s(t)\\
&=\Big(B_s(t)+\frac{\epsilon l^2}{z}\big(\eta_s(t)\Delta_d B_s(t)-B_s(t)\Delta_d\eta_s(t)\big)\Big)(1-\omega\delta t)+\theta n_s(t)P_s(t),
\end{split}
\end{equation}
where again we have used the notation of discrete Laplacian (\ref{23}).  By the same microscopic--macroscopic approach that leads to (\ref{24}), we obtain the following PDE
\begin{equation}\label{27}
\frac{\partial B}{\partial t}=\frac{\epsilon D}{z}\big(\eta \Delta B-B \Delta \eta \big)-\omega B+\kappa D\rho A.
\end{equation}
We want to point out that (\ref{27}) can be written into the following divergence form,
\[\frac{\partial B}{\partial t}=\frac{\epsilon D}{z}\nabla\cdot\left(\eta^2~\nabla\frac{B}{\eta}\right)-\omega B+\kappa D\rho A.\]
Moreover, (\ref{24}) and (\ref{27}) are equivalent if $\eta(A)$ is a constant.

\subsection{\textbf{Derivation of the burglars equation}}
It is assumed in \cite{SDP} that burglars must leave site $s$ for home after committing a crime at this site and they are removed from this location after time $t+\delta t$.  In this paper, we consider situations slightly general than \cite{SDP} and assume that criminal agents will take one of the three dispersal strategies at the next time step: (i) leave site $s$ with or without their hunting and then are removed from the system; (ii) stay at the current site; (iii) move to one of the neighbouring sites.  

To manifest the dispersal strategies above, we divide criminal agents at site $s$ into two groups: those who have committed burglaries during time $(t-\delta t,t]$ and those have not.  Each group of agents will take one of the three options at the next time step: stay at $s$ for further hunting, or move to a neighboring site of $s$, or go back home.  Define
\begin{equation}\label{28}
\begin{split}
&\lambda_1(\delta t): \text{probability that the burglar go home after burgling at $s$};\\
&\lambda_2(\delta t): \text{probability that the burglar go home without burgling at $s$}.
\end{split}
\end{equation}
It is very likely that the criminal agent who has burglarized at $s$ is still motivated for more looting goods at this site after a very short time period, therefore we assume that $\lambda_1(\delta t)\rightarrow \lambda_0 \in (0,1)$ as $\delta t\rightarrow0$; on the other hand, the criminal agent who has not burglarized are less likely to return home with nothing collected and to model this we assume that $\lambda_2(\delta t) \rightarrow 0$ as $\delta t\rightarrow0$.  For technical reasons, we assume that $\lambda_2(\delta t)=o(\delta t)$ as $\delta t\rightarrow0$ in our coming analysis.

If a criminal agent decides to move to one of the neighbouring sites on the grid, the dispersal is treated as random walk in \cite{SDP}, attracted by each of the neighbouring grids with dispersal probability
\[q_{s\rightarrow s'}(t)=\frac{A_{s'(t)}}{\sum\limits_{s'\sim s}A_{s'}(t)},\]
where $s'$ denotes any neighbouring site of $s$.  Our extension of this dispersal strategy is based on the following considerations.  First of all, if criminal agents choose to commit another burglar at the current site $s$ after burglarizing and if the current site is attractive enough, they would prefer staying for extra hunting than probing the neighbouring sites.  That being said, $q_{s\rightarrow s'}(t)$ needs to incorporate the information of the attractiveness at site $s$.

Second of all, it is assumed in \cite{SDP} that a criminal agent responds to the attractiveness $A_s(t)$ linearly, however this is based on the assumption that the agent has perfect information of the attractive $A_{s'}(t)$ of each neighbouring site, which might not be realistic.  See the definition of transition probability $q_{s\rightarrow s'}(t)$ above.  Therefore it is of our interest to put human response to physical stimulus into consideration and we assume that the criminal agent's perception of the attractiveness $A_s$ is not necessarily a linear function.   For the generality of our analysis, we denote $f(A_s)$ as the criminal agents' perception of the attractiveness of site $s$, where $f$ is a nonnegative function such that $f'(A)\geq0$ and $f(A)\leq A$ for all $A\geq0$.  One typical choice is a logarithmic type function, which resembles the Webner--Fecher's law on human perception to logarithmic of a physical stimulus.  Taking these discussions into account, we have that the dispersal probability of burglars moving from $s$ to $s'$ is
\begin{equation}\label{29}
q_{s\rightarrow s'}(t)=\frac{f(A_{s'},n_{s'})}{\sum\limits_{s''\in N(s)\cup\{s\}}f(A_{s''},n_{s''})},
\end{equation}
and the probability of burglars staying at $s$ is
\begin{equation}\label{210}
q_{s\rightarrow s}(t)=\frac{f(A_{s},n_{s})}{\sum\limits_{s''\in N(s)\cup\{s\}}f(A_{s''},n_{s''})},
\end{equation}
where $N(s)$ denotes the neighbouring sites of $s=(i,j)$.

We also like to point out that, in practice it is also realistic to assume that the criminal agents move to the neighbouring site $s'$ of the largest perceived attractiveness $f(A_{s'})$ with probability 1, especially when the difference between attractiveness at $s'$ and that of the rest site is ostensible.  For example, it is very likely that criminal agents would \emph{surely} move to the house that appears more attractive to them (e.g., with fancy cars, nice decoration, swimming pool, etc.) than the rest houses that appear less attractive, without having difficulty in choosing between the neighboring sites (unless the neighbouring sites are almost the same).  In this case, a realistic dispersal probability would be
\begin{equation*}
q_{s\rightarrow s'}(t)=
\left\{
\begin{array}{ll}
1,&\text{ if $f(A_{s'})=\max_{s^{''}\sim s} f(A_{s^{''}})$},\\
0,&\text{ otherwise}.\\
\end{array}
\right.
\end{equation*}
For the sake of our analysis we shall assume (\ref{29})--(\ref{210}) and derive the continuum models (\ref{12}) and (\ref{13}).

Criminal agents at site $s$ and time $t+\delta t$ consist of four types: the agents from site $s$ who burglarized at time $t$ but stayed there for the next time step, the agents from site $s$ who did not burglarize at time $t$ and stayed, the agents from neighbouring sites $s'$ who burglarized at time $t$, and those from neighbouring sites $s'$ who did not burglarized at time $t$.  Following the arguments in \cite{SDP} in the modeling of repeat victimization effect, we assume that the agents are created at a constant rate $\Gamma$ during time period of length $\delta t$.  Therefore the total population of criminal agents at site $s$ satisfies the following difference equation
\begin{equation}
\begin{split}
n_s(t+\delta t)=&\sum\limits_{s'\sim s}\frac{f(A_s,n_s)n_{s'}(1-P_{s'})(1-\lambda_2)}{\sum\limits_{s''\in N(s')\cup\{s'\}}f(A_{s''},n_{s''})}+\frac{f(A_s,n_s)n_{s}(1-P_{s})(1-\lambda_2)}{\sum\limits_{s''\in N(s)\cup\{s\}}f(A_{s''},n_{s''})}\\
+&\sum\limits_{s'\sim s}\frac{f(A_s,n_s)n_{s'}P_{s'}(1-\lambda_1)}{\sum\limits_{s''\in N(s')\cup\{s'\}}f(A_{s''},n_{s''})}+\frac{f(A_s,n_s)n_{s}P_{s}(1-\lambda_1)}{\sum\limits_{s''\in N(s)\cup\{s\}}f(A_{s''},n_{s''})}+\Gamma\delta t
\end{split}
\end{equation}
which can be simplified as
\begin{equation}\label{212}
\begin{split}
n_s(t+\delta t)=&\sum\limits_{s'\sim s}\frac{f(A_s,n_s)\big(n_{s'}P_{s'}(1-\lambda_1)+n_{s'}(1-P_{s'})(1-\lambda_2)\big)}{\sum\limits_{s''\in N(s')\cup\{s'\}}f(A_{s''},n_{s''})}\\
&+\frac{f(A_s,n_s)\big(n_{s}P_{s}(1-\lambda_1)+n_{s}(1-P_{s})(1-\lambda_2)\big)}{\sum\limits_{s''\in N(s)\cup\{s\}}f(A_{s''},n_{s''})}+\Gamma\delta t\\
=&\sum\limits_{s'\in N(s)\cup\{s\}}\frac{f(A_s,n_s)\big(n_{s'}P_{s'}(1-\lambda_1)+n_{s'}(1-P_{s'})(1-\lambda_2)\big)}{\sum\limits_{s''\in N(s')\cup\{s'\}}f(A_{s''},n_{s''})}+\Gamma\delta t.
\end{split}
\end{equation}
Furthermore, in terms of the discrete spatial Laplacian (\ref{23}), we can further simplify (\ref{212}) as
\begin{equation}\label{213}
\begin{split}
n_s(t+\delta t)=&f\Big(l^2\Delta_d\frac{nP(1-\lambda_1)+n(1-P)(1-\lambda_2)}{l^2\Delta_d f+(z+1)f}\\
&+z\frac{nP(1-\lambda_1)+n(1-P)(1-\lambda_2)}{l^2\Delta_d f+(z+1)f}\Big)\\
&+f\frac{nP(1-\lambda_1)+n(1-P)(1-\lambda_2)}{l^2\Delta_d f+(z+1)f}+\Gamma\delta t\\
=&f l^2\Delta_d\frac{nP(1-\lambda_1)+n(1-P)(1-\lambda_2)}{l^2\Delta_d f+(z+1)f}\\
&+(z+1)f \frac{nP(1-\lambda_1)+n(1-P)(1-\lambda_2)}{l^2\Delta_d f+(z+1)f}+\Gamma\delta t\\
\end{split}
\end{equation}
Denoting $\gamma=\Gamma/\delta t$ and $\frac{l^2}{\delta t}\rightarrow D$ as $\delta t,l\rightarrow 0$, we arrive at the continuum equation of criminal population density
\begin{equation}\label{214}
\frac{\partial\rho}{\partial t}=\frac{D}{z+1}\nabla\cdot(\nabla\rho-2\rho \nabla \log f(A))-\lambda_0\rho A+\gamma.
\end{equation}
We can easily observe in this continuum equation that the criminal agents direct their movements along the gradient of house attractiveness.

\subsubsection{Reaction--advection--diffusion systems with heterogeneous diffusion rate}
According to the analysis above, we arrive at the following two general systems
\begin{equation}\label{215}
\left\{
\begin{array}{ll}
\frac{\partial B}{\partial t}=\frac{\epsilon D}{z}\Delta(\eta B)-\omega B+\kappa D\rho A,\\
\frac{\partial\rho}{\partial t}=\frac{D}{z+1}\nabla\cdot(\nabla\rho-2\rho \nabla \log f)-\lambda_0\rho A+\gamma
\end{array}
\right.
\end{equation}
and
\begin{equation}\label{216}
\left\{
\begin{array}{ll}
\frac{\partial B}{\partial t}=\frac{\epsilon D}{z}\nabla\cdot\left(\eta^2~\nabla\frac{B}{\eta}\right)-\omega B+\kappa D\rho A,\\
\frac{\partial\rho}{\partial t}=\frac{D}{z+1}\nabla\cdot(\nabla\rho-2\rho \nabla \log f)-\lambda_0\rho A+\gamma.
\end{array}
\right.
\end{equation}
Due to transformations
\begin{equation*}
\tilde  A=A/\omega,~\tilde\rho=\frac{\kappa D}{\omega}\rho,~\tilde {\textbf{x}}=\sqrt{\frac{\omega(z+1)}{D}}\textbf{x},~\tilde t=\omega t.
\end{equation*}
and $\frac{z+1}{z}\eta(\omega\tilde A)=\tilde{\eta}(\tilde A)$ and $f(\omega\tilde A)=\tilde{f}(\tilde A)$,
systems (\ref{215}) and (\ref{216}) become (\ref{12}) and (\ref{13}) respectively, where the tildes are dropped there without causing any confusion.  It is the goal of our paper to investigate the formation of patterns in (\ref{215}) and (\ref{216}) through models (\ref{12}) and (\ref{13}).  In particular, we want to study the qualitative behaviors of stable steady states to these systems that model the crime data clustering phenomenon.  Without losing the generality of our analysis, we impose homogeneous Neumann boundary conditions for both systems.

\section{Linear stability analysis of homogeneous steady state}\label{section3}
We are interested in the formation of positive nonconstant steady states of (\ref{12}) and (\ref{13}) subject to (\ref{14}) with interesting patterns.  Our starting point is the stability analysis of the homogeneous steady state
\[(\bar A,\bar \rho)=\Big(A^0+\bar B,\frac{\bar B}{A^0+\bar B}\Big).\]

Linearizing (\ref{12}) around $(\bar A,\bar \rho)$, we have from simple calculations that stability of the homogeneous steady state is determined by the eigenvalues of the following matrix
\begin{equation}\label{31}
\mathcal H_k=
\left(
\begin{array}{cc}
-\epsilon\left(\eta(\bar A)+\eta\prime(\bar A)\bar B\right)\sigma_k+\bar\rho-1& \bar A  \\
\frac{2\bar\rho f'(\bar A)}{f(\bar A)}\sigma_k-\lambda_0\bar\rho& -\sigma_k-\lambda_0\bar A \\
\end{array}
\right),
\end{equation}
where $\sigma=\sigma_k>0$, $k=1,2,\cdots$ are the $k$--th eigenvalues of $-\Delta$ on $\Omega$ under the Neumann boundary conditions.  It is well known that the Neumann Laplacian has a discrete spectrum of infinitely many non-negative eigenvalues which form a strictly increasing sequence $0=\sigma_0<\sigma_1<\sigma_2<...<\sigma_k<...\rightarrow \infty$.  In the sequel, we assume that $\sigma_k$ is a simple eigenvalue and the eigenfunctions $\{\Phi_k\}_{k=0}^\infty$ form a complete orthonormal basis of $L^2(\Omega)$ with $\Phi_0\equiv \frac{1}{\vert \Omega\vert}$ and $\int_\Omega \Phi_k^2 dx=1$ for each $k\in \mathbb R^+$.  We have the instability result in the following proposition.
\begin{proposition}\label{proposition31}
Let $\sigma_k$ be the $k$--th Neumann eigenvalue of $-\Delta$.  The constant solution $(\bar A,\bar\rho)$ of (\ref{12}) is unstable if and only if
\begin{equation}\label{32}
\epsilon<\max_{k\in\mathbb{N^+}}\frac{\big(\frac{2\bar Bf'(\bar A)}{f(\bar A)}+\bar\rho-1\big)\sigma_k-\lambda_0\bar A}{(\eta(\bar A)+\eta'(\bar A)\bar B)(\sigma_k+\lambda_0\bar A)\sigma_k}.
\end{equation}
\end{proposition}

\begin{proof}
According to the principle of exchange of stability, see Theorem 5.2 in \cite{D} e.g., $(\bar A,\bar \rho)$ is stable if both eigenvalues of each $\mathcal H_k$ have negative real parts and it is unstable if $\mathcal H_{k}$ has an eigenvalue with positive real part for some $k \in \mathbb{N}^+$.  The characteristic polynomial of (\ref{31}) takes the form $g(\xi)=\xi^2+\text{Tr}\xi+\text{Det}$, where \[\text{Tr}=\epsilon\left(\eta(\bar A)+\eta\prime(\bar A)\bar B\right)\sigma_k+1-\bar\rho+\sigma_k+\lambda_0\bar A,\]
and
\[\text{Det}=\left(\epsilon\left(\eta(\bar A)+\eta\prime(\bar A)\bar B\right)\sigma_k+1-\bar\rho\right)\left(\sigma_k+\lambda_0\bar A\right)-\bar A\Big(\frac{2\bar\rho f'(\bar A)}{f(\bar A)}\sigma_k-\lambda_0\bar\rho\Big).\]
$\text{Tr}>0$ since $\bar \rho =\frac{\bar B}{A^0+\bar B}<1$, therefore $g(t)$ has one positive root if and only if $g(0)=\text{Det} <0$.  Then (\ref{32}) follows from straightforward calculations and this finishes the proof of Proposition \ref{proposition31}.
\end{proof}
By the same analysis we can prove the instability of the homogeneous steady state with respect to (\ref{13}) as follows.
\begin{corollary}\label{corollary31}
Suppose that $\eta(\bar A)>\eta'(\bar A)\bar B$.  The constant solution $(\bar A,\bar\rho)$ of (\ref{13}) is unstable if and only if
\begin{equation}\label{33}
\epsilon<\max_{k\in\mathbb{N^+}}\frac{(\frac{2\bar Bf'(\bar A)}{f(\bar A)}+\bar\rho-1)\sigma_k-\lambda_0\bar A}{(\eta(\bar A)-\eta'(\bar A)\bar B)(\sigma_k+\lambda_0\bar A)\sigma_k}.
\end{equation}
\end{corollary}
Our stability analysis above suggests that small intrinsic diffusion rate $\epsilon$ destroys the stability of the homogeneous steady state $(\bar A,\bar \rho)$ to (\ref{12}) and (\ref{13}).  Here $\epsilon$ interprets the strength of intrinsic near--repeat victimization effect, therefore when each site is sensitive to burglars the neighbouring sites and the near--repeat victimization effect is strong, clustering of data is not expected in the community.  It is assumed that $\eta'(A)>0$ for all $A>0$, i.e., neighbourhood with a larger house attractiveness has stronger near--repeat victimization effect.  Comparing (\ref{32}) and (\ref{33}), we see that when $\epsilon$ is between the maximum values defined in (\ref{32}) and (\ref{33}), $(\bar A,\bar \rho)$ loses its stability in the arrival--dependent model (\ref{13}), while it is still stable in the departure--dependent model (\ref{12}).  More detailed are included in Section \ref{section6} to discuss the differences between the departure--dependent and arrival--dependent models.

\section{Existence of nonconstant positive steady states}\label{section4}
In this section, we study the existence of nonconstant positive steady states of (\ref{12}) and (\ref{13}) under (\ref{14}), i.e., nonconstant positive solutions to the following quasi--linear elliptic systems
\begin{equation}\label{41}
\left\{
\begin{array}{ll}
\epsilon\Delta(\eta(A)(A-A^0))-A+A^0+\rho A=0,&x\in \Omega\\
\nabla\cdot(\nabla\rho-2\rho \nabla\log f(A))-\lambda_0\rho A+\lambda_0 \bar B=0,&x\in \Omega\\
\frac{\partial A}{\partial\mathbf{n}}=\frac{\partial \rho}{\partial\mathbf{n}}=0,&x\in\partial\Omega,
\end{array}
\right.
\end{equation}
and
\begin{equation}\label{42}
\left\{
\begin{array}{ll}
\epsilon\nabla\cdot\left(\eta^2(A)~\nabla\frac{A-A^0}{\eta(A)}\right)-A+A^0+\rho A=0, &x\in \Omega,\\
\nabla\cdot(\nabla\rho-2\rho\nabla \log f(A))-\lambda_0\rho A+\lambda_0\bar B=0, &x\in  \Omega,\\
\frac{\partial A}{\partial\mathbf{n}}=\frac{\partial \rho}{\partial\mathbf{n}}=0,&x\in\partial\Omega.
\end{array}
\right.
\end{equation}

We shall perform rigorous bifurcation analysis due to Crandall--Rabinowitz \cite{CR} to establish nonconstant positive solutions for (\ref{41}), while the same analysis can be carried out for (\ref{42}).  Taking $\epsilon$ as the bifurcation parameter, we introduce the operator
\begin{equation}\label{43}
\mathcal{F}=\left(
\begin{array}{ll}
\epsilon\Delta(\eta(A)(A-A^0))-A+A^0+\rho A \\
\nabla\cdot(\nabla\rho-2\rho \nabla \log f(A))-\lambda_0\rho A+\lambda_0 \bar B\\
\end{array}
\right)
\end{equation}
from $\mathcal{\mathcal{X}}\times\mathcal{\mathcal{X}}\times \mathbb R^+$ to $\mathcal{Y}\times\mathcal{Y}$, where $\mathcal{X}$ is Sobolev space $\mathcal{X}=\{w\in W^{2,p}(\Omega)\big{|}\frac{\partial w}{\partial\mathbf{n}}=0,x\in\partial\Omega\}$ and $\mathcal{Y}=L^p(\Omega)$ for some $p>N$.  We want to point out that here the boundary condition makes sense since $W^{2,p}(\Omega)\hookrightarrow C^{1+\alpha}(\bar \Omega)$ for some $\alpha>0$ if $p>N$.  (\ref{41}) is equivalent to $\mathcal{F}(A,\rho,\epsilon)=0$ for $(A,\rho,\epsilon)\in\mathcal{X}\times\mathcal{X}\times \mathbb R^+$.  It is easy to see that $\mathcal{F}$ is a continuously differentiable mapping from $\mathcal{X}\times\mathcal{X}\times \mathbb R^+$ to $\mathcal{Y}\times\mathcal{Y}$ and $\mathcal{F}(\bar A,\bar\rho,\epsilon)=0$ for any $\epsilon\in \mathbb R^+$; moreover, for any fixed $(\hat A,\hat\rho)\in\mathcal{X}\times\mathcal{X}$, the Fr\'echet derivative of $\mathcal{F}$ is given by
\begin{equation}\label{44}
D_{(A,\rho)}\mathcal{F}(\hat A,\hat\rho,\epsilon)(A,\rho)=\left(
\begin{array}{cc}
\vspace{0.2in}
\epsilon\Delta((\eta(\hat A)+\eta'(\hat A)(\hat A-A^0))A)+(\hat\rho-1)A+\hat A\rho \\
\nabla\cdot\big(\nabla\rho-\frac{2\hat\rho f'(\hat A)}{f(\hat A)}\nabla A+\frac{2\hat\rho f'(\hat A)\nabla f(\hat A)}{f^2(\hat A)}A-\frac{2\hat\rho\nabla f'(\hat A)}{f(\hat A)}A\\
-\frac{2 \rho}{f(\hat A)}\nabla f(\hat A)\big)-\lambda_0\hat A\rho-\lambda_0\hat\rho A
\end{array}
\right);
\end{equation}
furthermore, by rewriting (\ref{44}) as
\[D_{(A,\rho)}\mathcal{F}(\hat A,\hat\rho,\epsilon)(A,\rho)=I_1\left(
\begin{array}{ll}
\Delta A\\
\Delta\rho
\end{array}
\right)+I_2\left(
\begin{array}{ll}
\nabla A\\
\nabla\rho
\end{array}
\right)+I_3\left(
\begin{array}{ll}
A\\
\rho
\end{array}
\right),\]
with
\[I_1=\left(
\begin{array}{cc}
\epsilon\left(\eta(\hat A)+\eta\prime(\hat A)(\hat A-A^0)\right)& ~~~~0 ~~~ \\
-\frac{2\hat\rho f'(\hat A)}{f(\hat A)}& ~~~~1 ~~~\\
\end{array}
\right),\]
\[I_2=\left(
\begin{array}{cc}
2\epsilon\nabla(\eta(\hat A)+\eta\prime(\hat A)(\hat A-A^0))&~~~ 0 ~~~ \\
\frac{2\hat\rho f'(\hat A)\nabla f(\hat A)}{f^2(\hat A)}-\frac{2\hat\rho\nabla f'(\hat A)}{f(\hat A)}-\nabla\frac{2\hat\rho f'(\hat A)}{f(\hat A)}& ~~~\frac{-2\nabla f(\hat A)}{f(\hat A)} ~~~\\
\end{array}
\right),\]
\[I_3=\left(
\begin{array}{cc}
\epsilon\Delta(\eta(\hat A)+\eta\prime(\hat A)(\hat A-A^0))+\hat\rho-1& ~~~\hat A  \\
\nabla\cdot(\frac{2\hat\rho f'(\hat A)\nabla f(\hat A)}{f^2(\hat A)}-\frac{2\hat\rho\nabla f'(\hat A)}{f(\hat A)})-\lambda_0\hat\rho& ~~~-2\nabla\cdot\frac{\nabla f(\hat A)}{f(\hat A)}-\lambda_0\hat A \\
\end{array}
\right),\]
we see that (\ref{44}) is a linear and compact operator according to standard elliptic regularity and Sobolev embeddings.  On the other hand, matrix $I_1$, defining the principal part of $D_{(A,\rho)}\mathcal{F}(\hat A,\hat\rho,\epsilon)$, has two positive eigenvalues, therefore $D_{(A,\rho)}\mathcal{F}(\hat A,\hat\rho,\epsilon)$ is a Fredholm operator with 0 index by Corollary 2.11 or Remark 3.4 of theorem 3.3 in Shi and Wang \cite{SDP}.

The necessary condition for bifurcation at $(\bar A,\bar\rho)$ is $\mathcal{N}\left(D_{(A,\rho)}\mathcal{F}(\bar A,\bar\rho,\epsilon)\right)\neq\{(0,0)\}$, where $\mathcal{N}$ denotes the null set and
\begin{equation}\label{45}
D_{(A,\rho)}\mathcal{F}(\bar A,\bar\rho,\epsilon)(A,\rho)=\left(
\begin{array}{ll}
\epsilon(\eta(\bar A)+\eta'(\bar A)\bar B)\Delta A+(\bar\rho-1)A+\bar A\rho \\
~~~\Delta\rho-\frac{2\bar\rho f^{\prime}(\bar A)}{f(\bar A)}\Delta A-\lambda_0\bar A\rho-\lambda_0\bar\rho A
\end{array}
\right).
\end{equation}
To show this condition, we choose some $(A,\rho)\in D_{(A,\rho)}\mathcal{F}(\bar A,\bar\rho,\epsilon)$ and substituting their eigen-expansions $A=\sum_{k=0}^{\infty} S_k \Phi_k(x),\rho=\sum_{k=0}^{\infty}  T_k \Phi_k(x)$ into (\ref{45}) to collect
\begin{equation}\label{46}
\left(
\begin{array}{cc}
-\epsilon(\eta(\bar A)+\eta\prime(\bar A)\bar B)\sigma_k+\bar\rho-1& \bar A  \\
\frac{2\bar\rho f'(\bar A)}{f(\bar A)}\sigma_k-\lambda_0\bar\rho& -\sigma_k-\lambda_0\bar A \\
\end{array}
\right)
\left(
\begin{array}{cc}
S_k\\
T_k
\end{array}
\right)=0,
\end{equation}
then for each $k \in \mathbb R^+$, (\ref{46}) has nontrivial solutions if and only if the coefficient matrix is singular, i.e.,
\begin{equation}\label{47}
\epsilon=\bar \epsilon_k=\frac{\big(\frac{2\bar Bf'(\bar A)}{f(\bar A)}+\bar\rho-1\big)\sigma_k-\lambda_0\bar A}{(\eta(\bar A)+\eta'(\bar A)\bar B)(\sigma_k+\lambda_0\bar A)\sigma_k},~k\in\mathbb{N}^+,
\end{equation}
and this is the necessary condition for bifurcations to occur at $(\bar A,\bar \rho)$.  If the eigenvalue $\sigma_k$ is simple, it follows that $\mathcal{N}(D_{(A,\rho)}\mathcal{F}(\bar{A},\bar{\rho},\bar \epsilon_{k}))$ is one--dimensional and its basis is spanned by $\left\{ (\bar{A}_{k}, \bar{\rho}_{k}) \right\}$ with
\begin{equation}\label{48}
\bar{A}_{k}=Q_k\Phi_k,~\bar{\rho}_{k}=\Phi_k,
\end{equation}
where $\Phi_k$ is the $k$--th Neumann Laplacian eigenfunction and
\begin{equation}\label{49}
Q_k=\frac{\sigma_k+\lambda_0\bar A}{\frac{2\bar\rho f'(\bar A)}{f(\bar A)}\sigma_k-\lambda_0\bar\rho}, k \in \mathbb{N}^+.
\end{equation}
The fact that $D_{(A,\rho)}\mathcal{F}(\bar A,\bar{\rho},\epsilon)$ is Fredholm with index 0 implies that dim $\mathcal N(D_{(A,\rho)}\mathcal{F}(\bar A,\bar{\rho},\epsilon))$ equals codim $\mathcal R(D_{(A,\rho)}\mathcal{F}(\bar A,\bar{\rho},\epsilon))=1$.

We now verify that local bifurcation does occur at $(\bar A,\bar \rho,\bar \epsilon_{k})$ in the following theorem, which establishes nonconstant positive solutions to (\ref{41}).
\begin{theorem}\label{theorem41}
Let $\sigma_k$ be the $k$--th Neumann Laplace eigenvalue such that $\sigma_k \neq \frac{\lambda_0   f(\bar A)}{2f'(\bar A)}$ and $\frac{\lambda_0\bar A+(1-\bar\rho)\sigma_k}{2\sigma_k\bar B} <\frac{f'(\bar A)}{f(\bar A)}$ for each $k\in \mathbb N^+$.  Moreover suppose that
\begin{equation}\label{410}
\Big(\frac{2\bar B f'(\bar A)}{f(\bar A)}+\bar\rho-1\Big)\sigma_k\sigma_j\neq \lambda_0\bar A(\sigma_k+\sigma_j+\lambda_0\bar A)
\end{equation}
for all $j\neq k\in \mathbb N^+$.  Then for each $k\in \mathbb N^+$, (\ref{41}) admits solutions around $(\bar A,\bar \rho,\bar \epsilon_k)$ that consist precisely of the continuously smooth curve $\Gamma_{k}(s)=\{(A_{k}(s,x), \rho_{k}(s,x), \epsilon_{k}(s))\}$, $s\in(-\delta,\delta)$, where
\begin{equation}\label{411}
\epsilon_{k}(0)=\bar \epsilon_{k}+o(s),~(A_{k}(s,x),\rho_{k}(s,x))=(\bar A,\bar \rho)+s(Q_{k},1)\Phi_k+o(s);
\end{equation}
moreover, $(A_{k}(s,x),\rho_{k}(s,x))-(\bar A,\bar \rho)-s(Q_{k},1)\Phi_k$ is in the closed complement of the null space $\mathcal{Z}$ of $D_{(A,\rho)}\mathcal{F}(\bar{A},\bar{\rho},\bar \epsilon_{k})$, which is explicitly given by
\begin{equation}\label{412}
\mathcal{Z}=\big\{(A,\rho)\in \mathcal{X} \times \mathcal{X}~ \big \vert \int_{\Omega} A \bar A_{k}+\rho\bar \rho_{k} dx=0\big\},
\end{equation}
where $(\bar A_{k},\bar \rho_{k})$ is defined in (\ref{48}).
\end{theorem}

\begin{proof} To make use of the local bifurcation theory in \cite{CR}, we have verified all but the so-called transversality condition: $\frac{d}{d\epsilon}D_{(A,\rho)}\mathcal{F}(\bar A,\bar \rho,\epsilon)(\bar A_k,\bar \rho_k)\big|_{\epsilon=\bar \epsilon_k}\not\in \mathcal R\left(D_{(A,\rho)}\mathcal{F}(\bar A,\bar\rho,\bar \epsilon_k)\right)$, where
\[\frac{d}{d\epsilon}D_{(A,\rho)}\mathcal{F}(\bar A,\bar\rho,\epsilon)(\bar A_k,\bar \rho_k)\big|_{\epsilon=\bar \epsilon_k}=\left(
\begin{array}{cc}
(\eta(\bar A)+\eta'(\bar A)\bar B)\Delta \bar A_k\\
0
\end{array}
\right).\]
We argue by contradiction and assume that there exists $(\tilde A,\tilde\rho)\in \mathcal{X}\times \mathcal{X}$ that satisfies
\begin{equation*}
\left\{
\begin{array}{ll}
\epsilon_k(\eta(\bar A)+\eta'(\bar A)\bar B)\Delta\tilde A+(\bar\rho-1)\tilde A+\bar A\tilde\rho=(\eta(\bar A)+\eta'(\bar A)\bar B)\Delta \bar A_k, &x\in\Omega,\\
\Delta\tilde\rho-\frac{2\bar \rho f'(\bar A)}{f(\bar A)}\Delta\tilde A-\lambda_0\bar\rho\tilde A-\lambda_0\bar A\tilde\rho=0, &x\in\Omega,\\
\frac{\partial\tilde A}{\partial\mathbf{n}}=\frac{\partial\tilde \rho}{\partial\mathbf{n}}=0, &x\in\partial\Omega.
\end{array}
\right.
\end{equation*}
In light of the eigen--expansions $\tilde A=\sum_{k=0}^{\infty}\tilde S_k\Phi_k$ and $\tilde\rho=\sum_{k=0}^{\infty}\tilde T_k\Phi_k$, we collect from the system above that
\begin{equation}\label{413}
\left(
\begin{array}{cc}
-\bar \epsilon_k(\eta(\bar A)+\eta\prime(\bar A)\bar B)\sigma_k+\bar\rho-1& \bar A  \\
\frac{2\bar\rho f'(\bar A)}{f(\bar A)}\sigma_k-\lambda_0\bar\rho& -\sigma_k-\lambda_0\bar A \\
\end{array}
\right)
\left(
\begin{array}{cc}
\tilde S_k\\
\tilde T_k
\end{array}
\right)=
\left(
\begin{array}{cc}
-(\eta(\bar A)+\eta\prime(\bar A)\bar B)\sigma_kS_k\\
0
\end{array}
\right).
\end{equation}
$k=0$ is ruled out since $\sigma_0=0$.  For $k\in \mathbb N^+$, the coefficient matrix of (\ref{413}) is singular thanks to (\ref{47}), however this is impossible since the right hand side of (\ref{413}) is nonzero and we reach a contradiction.  This verifies the transversality condition and the remaining statements follow from Theorem 1.7 in \cite{CR}.
\end{proof}
\begin{remark}\label{remark41}
Define $\Gamma^{\pm}_{k}(s)=\{(A_{k}(s,x), \rho_{k}(s,x), \epsilon_{k}(s)), \pm s\in(0,\delta)\}$ and let $\mathcal C$ be any connected component of solution set to (\ref{41}) over $\mathcal X \times \mathcal X \times \mathbb R^+$.  Denote $\mathcal C^\pm$ as the connected component of $\mathcal C \backslash \Gamma^\mp_{k}(s)$ which contains $\Gamma^\pm_{k}(s)$.   According to Theorem 44 in \cite{SW}, each $\mathcal C^\pm$ satisfies one of the three alternatives: (i). it is not compact; (ii). it contains a point $(\bar A,\bar \rho,\bar \epsilon^*)$ with $\bar \epsilon^* \neq \bar \epsilon_k$, for any $k\in\mathbb N^+$; (iii).  it contains a point $(A+\tilde A,\rho+\tilde \rho,\epsilon)$ with $(\tilde A,\tilde \rho),\neq (0,0),\in \mathcal Z$.  By the same calculations that lead to the transversality condition, we can easily rule out case (iii).  However, it is a mathematical challenging problem to characterize and determine when case (i) or (ii) may happen in order to extend the local bifurcation branches to global.  Global bifurcation results are very important in studying the qualitative behaviors of nonconstant positive steady states $(A(x),\rho(x))$ to (\ref{41}) when $\epsilon$ is away from $\max_{k\in \mathbb N^+}\bar \epsilon_k$, for which $(\bar A_k(s,x), \bar \rho_k(s,x))$ is unstable.  See Theorem \ref{theorem51}.
\end{remark}

Similarly as above we can show the existence of nonconstant positive solutions to (\ref{42}).
\begin{theorem}\label{theorem42}Denote
\begin{equation}\label{414}
 \bar \epsilon_k=\frac{(\frac{2\bar Bf'(\bar A)}{f(\bar A)}+\bar\rho-1)\sigma_k-\lambda_0\bar A}{(\eta(\bar A)-\eta'(\bar A)\bar B)(\sigma_k+\lambda_0\bar A)\sigma_k}.
\end{equation}
Suppose that $\bar \epsilon_k>0$ and $\bar \epsilon_k\neq \bar \epsilon_j$ for any $k\neq j\in \mathbb N^+$ .  Then (\ref{42}) has nonconstant positive solutions $(A_k(s,x),\rho_k(s,x))$ on curve $\Gamma_k(s)$ around $(\bar A,\bar \rho, \bar \epsilon_k)$; moreover $\Gamma_k(s)$ satisfies all the properties in Theorem \ref{theorem41}.
\end{theorem}

\section{Stability analysis of the bifurcating solutions around $(\bar A, \bar\rho,\bar \epsilon_k)$}\label{section5}
We now proceed to investigate the stability or instability of the spatially inhomogeneous solution $(A_k(s,x),\rho_k(s,x))$ established in Theorem \ref{theorem41}. To this end, we apply the results from Crandall--Rabinowitz \cite{CR2} on the linearized stability of bifurcating solutions with an analysis of the spectrum of system (\ref{41}). Stability here refers to the stability of the inhomogeneous patterns taken as an equilibrium of (\ref{41}).  According to Theorem 3.2 of \cite{CR2}, we can write the following asymptotic expansions of the bifurcating solution to (\ref{41})
\begin{equation}\label{51}
\left\{
\begin{array}{ll}
A_k(s,x)=\bar A+sQ_k\Phi_k+s^2\psi_1+s^3\psi_2+o(s^3),\\
\rho_k(s,x)=\bar\rho+s\Phi_k+s^2\varphi_1+s^3\varphi_2+o(s^3),\\
\epsilon_k(s)=\bar\epsilon_k+K_1s+K_2s^2+o(s^2),
\end{array}
\right.
\end{equation}
where $(\psi_i,\varphi_i)\in\mathcal{Z}$ for $i=1,2$ and $o(s^3)$ terms in $A_k(s,x)$ and $\rho_k(s,x)$ are taken in $W^{2,p}$--topology.  First of all, we evaluate $K_1$ and $K_2$ in the following proposition.  Here and in the sequel $K_1$ and $K_2$ depend on $k$ and we have skipped this index for simplicity of notation.
\begin{proposition}\label{proposition51}
Suppose that all conditions in Theorem \ref{theorem41} hold.  Then $K_1$ given in (\ref{51}) satisfies
\begin{equation}\label{52}
\begin{split}
&\big(\eta(\bar A)+\eta'(\bar A)\bar B\big)Q_k\sigma_kK_1\\
=&\Bigg(\frac{(\bar\rho-1-(\eta(\bar A)+\eta'(\bar A)\bar B)\bar\epsilon_k\sigma_k-\bar AQ_k)(\lambda_0Q_k-(\frac{f'(\bar A)}{f(\bar A)}+\bar\rho (\frac{f''(\bar A)}{f(\bar A)}-\frac{f'(\bar A)^2}{f(\bar A)^2})Q_k)Q_k\sigma_k)}{\frac{2\bar\rho f'(\bar A)}{f(\bar A)}\sigma_k-\lambda_0\bar\rho+(\sigma_k+\lambda_0\bar A)Q_k}\\
&+(Q_k-(\eta'(\bar A)+\frac{1}{2}\eta''(\bar A)\bar B)Q_k^2\bar\epsilon_k\sigma_k)\Bigg)\int_{\Omega}\Phi_k^3dx.
\end{split}
\end{equation}
\end{proposition}
$K_1$ is determined by $\int_{\Omega}\Phi_k^3dx$ and the system parameters.  In particular, if $\Omega$ is a finite 1D interval or multi-D rectangle, $\int_{\Omega}\Phi_k^3dx=0$ hence $K_1=0$.  In this case, we need to evaluate $K_2$ which is given below.
\begin{proposition}\label{proposition52}
Suppose all conditions in Theorem \ref{theorem41} hold.  If $K_1=0$ in (\ref{51}), then $K_2$ satisfies
\begin{equation}\label{53}
\begin{split}
&(\eta(\bar A)+\eta'(\bar A)\bar B)Q_k\sigma_k K_2\\
=&\Big(\bar\rho-1-(\eta(\bar A)+\eta'(\bar A)\bar B)\bar\epsilon_k\sigma_k\Big)\int_{\Omega}\psi_2\Phi_kdx+\bar A\int_{\Omega}\varphi_2\Phi_kdx\\
&+\Big(1-(2\eta'(\bar A)+\eta''(\bar A)\bar B)Q_k\bar\epsilon_k\sigma_k\Big)\int_{\Omega}\Phi_k^2\psi_1dx\\
&+Q_k\int_{\Omega}\Phi_k^2\varphi_1dx-\Big(\frac{\eta''(\bar A)}{2}+\frac{\eta'''(\bar A)}{6}\bar B\Big)Q_k^3\bar\epsilon_k\sigma_k\int_{\Omega}\Phi_k^4dx.
\end{split}
\end{equation}
\end{proposition}
Similarly, to determine the stability of the bifurcating solutions to (\ref{42}), we can write their expansions as in (\ref{51}), then by the same calculations that lead to (\ref{52}) and (\ref{53}), we can find $K_1$ and $K_2$ there in the following results.
\begin{corollary}\label{corollary52}
Let (\ref{51}) be the asymptotic solutions of (\ref{42}).  Suppose that all conditions in Theorem \ref{theorem42} are satisfied.  Then we have that
\begin{equation}\label{54}
\begin{split}
&(\eta(\bar A)-\eta'(\bar A)\bar B)Q_k\sigma_k K_1\\
=&\Bigg(\frac{\big(\bar\rho-1-(\eta(\bar A)-\eta'(\bar A)\bar B)\bar\epsilon_k\sigma_k-\bar AQ_k\big)\big(\lambda_0Q_k-(\frac{f'(\bar A)}{f(\bar A)}+\bar\rho (\frac{f''(\bar A)}{f(\bar A)}-\frac{f'(\bar A)^2}{f(\bar A)^2})Q_k)Q_k\sigma_k\big)}{\frac{2\bar\rho f'(\bar A)}{f(\bar A)}\sigma_k-\lambda_0\bar\rho+(\sigma_k+\lambda_0\bar A)Q_k}\\
&+(Q_k+\frac{1}{2}\eta''(\bar A)\bar BQ_k^2\bar\epsilon_k\sigma_k)\Bigg{)}\int_{\Omega}\Phi_k^3dx;
\end{split}
\end{equation}
moreover, if $K_1=0$, $K_2$ satisfies
\begin{equation}\label{55}
\begin{split}
&(\eta(\bar A)-\eta'(\bar A)\bar B)Q_k\sigma_k K_2\\
=&\Big(\bar\rho-1-(\eta(\bar A)-\eta'(\bar A)\bar B)\bar\epsilon_k\sigma_k\Big)\int_{\Omega}\psi_2\Phi_kdx+\bar A\int_{\Omega}\varphi_2\Phi_kdx\\
&+\Big(1+\eta''(\bar A)\bar BQ_k\bar\epsilon_k\sigma_k\Big)\int_{\Omega}\Phi_k^2\psi_1dx+Q_k\int_{\Omega}\Phi_k^2\varphi_1dx\\
&+\frac{1}{6}\Big(\eta''(\bar A)+\eta'''(\bar A)\bar{B}\Big)Q_k^3\bar\epsilon_k\sigma_k\int_{\Omega}\Phi_k^4dx.
\end{split}
\end{equation}
\end{corollary}

Bifurcation branch $\Gamma_k(s)$ is transcritical if $K_1\neq0$ and is pitchfork if $K_1=0$ and $K_2\neq0$.  Indeed, as we shall see in the coming theorem, if $K_1\neq0$ the sign of $K_1$ determines the stability of $(A_k(s,x),\rho_k(s,x))$, and if $K_1=0$ we need to determine the sign of $K_2$.  Now we present the stability of the bifurcating solution $(A_k(s,x),\rho_k(s,x))$ in the following theorem.
\begin{theorem}\label{theorem51}
 Suppose that all conditions in Theorem \ref{theorem41} hold.  Assume that $\bar \epsilon_{k_0}=\max_{k\in \mathbb N^+}\bar \epsilon_k$.  Then for all $k\neq k_0$, the steady state $(A_k(s,x),\rho_k(s,x))$ is always unstable for $s\in(-\delta,\delta)$.  If $K_1<0$, $(A_{k_0}(s,x),\rho_{k_0}(s,x))$ is asymptotically stable for $s\in(0,\delta)$ and is unstable for $s\in(-\delta,0)$; if $K_1>0$, $(A_{k_0}(s,x),\rho_{k_0}(s,x))$ is asymptotically stable for $s\in(-\delta,0)$ and is unstable for $s\in(0,\delta)$; moreover, if $K_1=0$, then $(A_{k_0}(s,x),\rho_{k_0}(s,x))$ is asymptotically stable for $s\in(-\delta,\delta)$ if $K_2<0$ and is unstable if $K_2>0$.
\end{theorem}
Transcritical and pitchfork bifurcations are schematically presented in Figure \ref{fig0} to illustrate the stability results in Theorem \ref{theorem51}.
\begin{figure}[h!]
\centering
\minipage{.22\textwidth}\centering
  \includegraphics[width=1.3in]{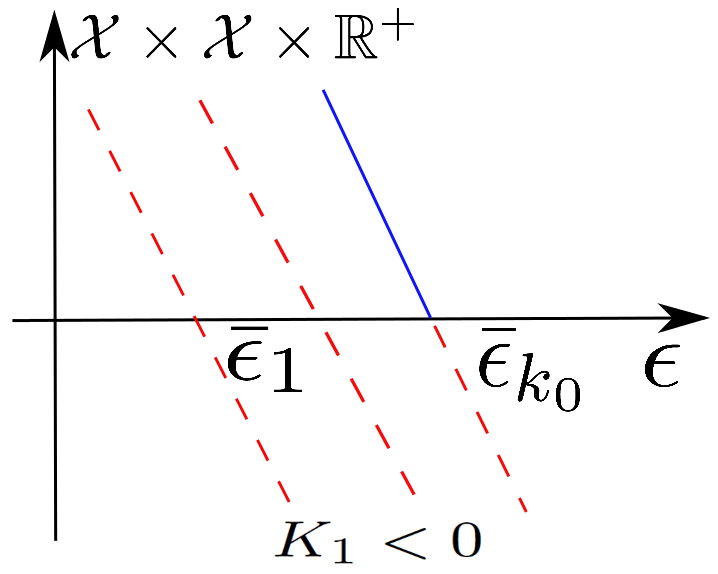}\caption*{Sub-transcritical}
\endminipage
\minipage{.22\textwidth}\centering
  \includegraphics[width=1.3in]{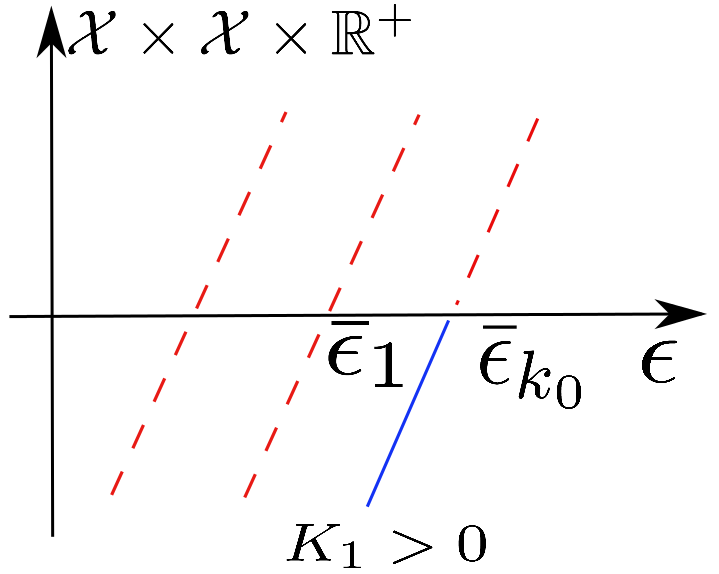}\caption*{Super-transcritical}
\endminipage
\minipage{.22\textwidth} \centering
  \includegraphics[width=1.3in]{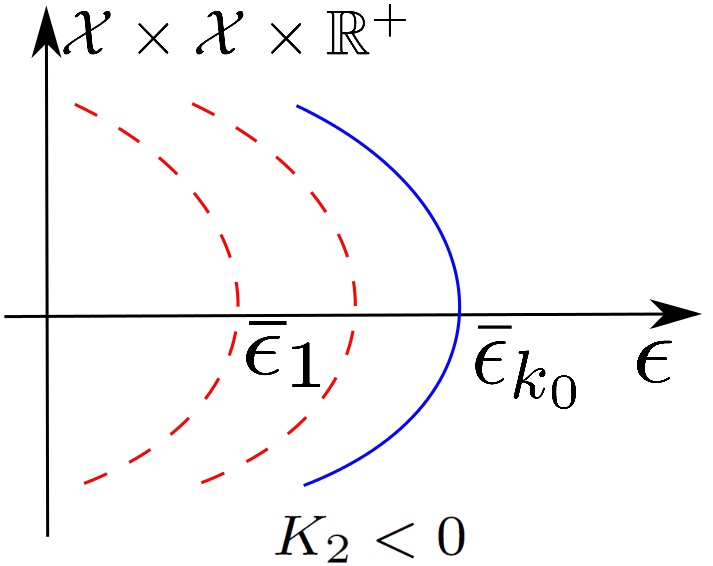}\caption*{Sub-pitchfork}
\endminipage
\minipage{.22\textwidth}\centering
  \includegraphics[width=1.3in]{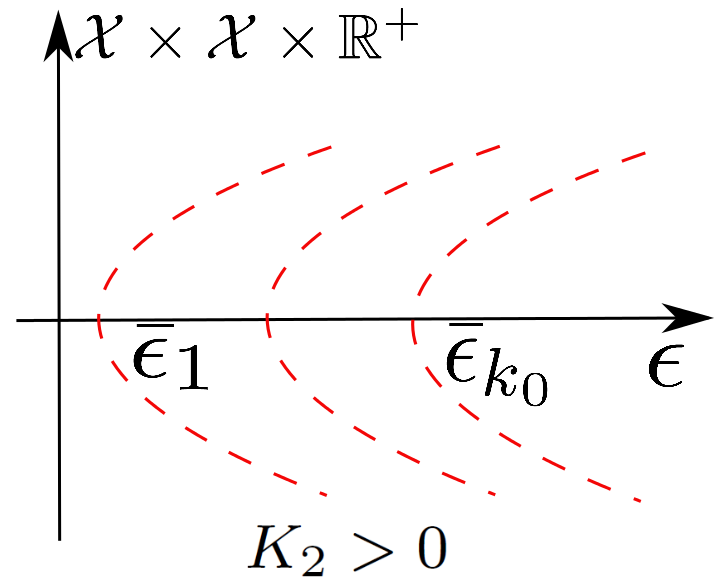}\caption*{Super-pitchfork}
\endminipage
\caption{Transcritical and pitchfork bifurcations to system (\ref{41}).  Stable bifurcation branches are presented in solid lines and unstable branches are in dashed lines.  The only stable branch must be the portion of the rightmost one that turns to the left.}\label{fig0}
\end{figure}
\begin{proof}
To study the stability of $(A_k(s,x),\rho_k(s,x))$, we linearize (\ref{12}) around this steady state and obtain the eigenvalue problem
\begin{equation}\label{56}
D_{(A,\rho)}\mathcal F (A_{k}(s,x),\rho_{k}(s,x),\bar \epsilon_{k}(s))(A,\rho)=\mu_k(s) (A,\rho),~(A,\rho)\in \mathcal X \times \mathcal X.
\end{equation}

We first show that $(A_k(s,x),\rho_k(s,x))$, $s\in(-\delta,\delta)$ is always unstable for all $k\neq k_0$ and to this end we shall only need to show that eigenvalue $\mu_k(s)$ has positive real part.  According to Corollary 1.13 in \cite{CR2}, $\mu_k=\mu_k(s)$ is a smooth function of $s: (-\delta,\delta) \rightarrow \mathbb R$.  Sending $s\rightarrow 0$, we see that (\ref{56}) becomes
\begin{equation}\label{57}
\left\{
\begin{array}{ll}
\bar \epsilon_k(\eta(\bar A)+\eta'(\bar A)\bar B)\Delta A+(\bar\rho-1) A+\bar A\rho=\mu_k(0) A, &x\in\Omega,\\
\Delta\rho-\frac{2\bar \rho f'(\bar A)}{f(\bar A)}\Delta A-\lambda_0\bar\rho A-\lambda_0\bar A\rho=\mu_k(0) \rho, &x\in\Omega,\\
\frac{\partial A}{\partial\mathbf{n}}=\frac{\partial \rho}{\partial\mathbf{n}}=0, &x\in\partial\Omega.
\end{array}
\right.
\end{equation}
Multiplying (\ref{57}) by $\Phi_k$ and integrating it over $\Omega$ by parts, we have that
\begin{equation*}
\left(
\begin{array}{cc}
-\bar \epsilon_k(\eta(\bar A)+\eta\prime(\bar A)\bar B)\sigma_k+\bar\rho-1-\mu_k(0)& \bar A  \\
\frac{2\bar\rho f'(\bar A)}{f(\bar A)}\sigma_k-\lambda_0\bar\rho& -\sigma_k-\lambda_0\bar A-\mu_k(0) \\
\end{array}
\right)
\left(
\begin{array}{cc}
\int_\Omega A \Phi_k dx\\
\int_\Omega \rho \Phi_k dx
\end{array}
\right)=
\left(
\begin{array}{cc}
0\\
0
\end{array}
\right),
\end{equation*}
therefore $\mu_k(0)$ is an eigenvalue of the characteristic polynomial $g_k(\mu)=\mu^2+\text{Tr}_k\mu+\text{Det}_k$, where
 \[\text{Tr}_k=\bar \epsilon_k\left(\eta(\bar A)+\eta\prime(\bar A)\bar B\right)\sigma_k+1-\bar\rho+\sigma_k+\lambda_0\bar A\]
 and
\[\text{Det}_k=\left(\bar \epsilon_k\left(\eta(\bar A)+\eta\prime(\bar A)\bar B\right)\sigma_k+1-\bar\rho\right)\left(\sigma_k+\lambda_0\bar A\right)-\bar A\big(\frac{2\bar\rho f'(\bar A)}{f(\bar A)}\sigma_k-\lambda_0\bar\rho\big).\]
$\text{Tr}_k>0$ since $\bar \rho =\frac{\bar B}{A^0+\bar B}<1$ and $\text{Det}_k<0$ since $\bar \epsilon_k<\bar \epsilon_{k_0}=\max_{k\in \mathbb N^+} \bar \epsilon_k$, therefore for any $k\neq k_0$, the characteristic polynomial $g_k(\mu)=0$ always has one positive root hence (\ref{57}) must have an eigenvalue $\mu_k(0)$ that has a positive real part.  From the standard eigenvalue perturbation theory in \cite{Ka}, (\ref{56}) always has a positive root $\mu_k(s)$ for each $k\neq k_0$ when $s$ is small and this verifies the instability of $(A_k(s,x),\rho_k(s,x))$ around $(\bar A,\bar\rho_k)$.

To show that $(A_{k_0}(s,x),\rho_{k_0}(s,x))$ on bifurcation branch $\Gamma_{k_0}(s)$ are asymptotically stable, it suffcies to show that the real parts of all eigenvalues of $D_{(A,\rho)}\mathcal F (A_{k_0}(s,x),\rho_{k_0}(s,x),\bar \epsilon_{k_0}(s))$ in (\ref{56}) are negative.  First of all, applying the same arguments that lead to the Fredholmness for Theorem \ref{theorem41}, we can show that 0 is a K--simple eigenvalue of $D_{(A,\rho)}\mathcal F (\bar A,\bar \rho,\bar \epsilon_{k_0})$--see Definition 1.2 in \cite{CR2}.  According to Corollary 1.13 in \cite{CR2}, there exists an interval $I$ with $\bar \epsilon_{k_0}\in I$ and continuously differentiable functions $\nu(\epsilon): I \rightarrow \mathbb R$ and $\mu_{k_0}(s): (-\delta,\delta) \rightarrow \mathbb R$ such that $\nu=\nu(\epsilon)$ is a real eigenvalue of
 \begin{equation}\label{58}
 D_{(A,\rho)}\mathcal F (\bar A,\bar \rho, \epsilon)(A,\rho)=\nu(\epsilon)(A,\rho),(A,\rho)\in \mathcal X \times \mathcal X
 \end{equation}
with $\nu(\bar \epsilon_{k_0})=0$ and $\mu_{k_0}=\mu_{k_0}(s)$ is an eigenvalue of (\ref{56}) with $\mu_{k_0}(0)=0$ ; moreover $\nu(\bar \epsilon_{k_0})$ is the only eigenvalue of (\ref{58}) for any fixed neighbourhood of the origin of the complex plane;  furthermore, the eigenfunction of (\ref{58}) depends on $\epsilon$ smoothly and can be written as $(A(\epsilon,x),\rho(\epsilon,x))$, which is uniquely determined by $(A(\bar \epsilon_{k_0},x),\rho(\bar \epsilon_{k_0},x))=(\bar A_{k_0},\bar \rho_{k_0})$ and $(A(\epsilon,x),\rho(\epsilon,x))-(\bar A_{k_0},\bar \rho_{k_0}) \in \mathcal Z$.

We now proceed to evaluate the sign of $\mu_{k_0}(s)$ for $s\in(-\delta,\delta)$.  Differentiating (\ref{58}) with respect to $\epsilon$ and putting $ \epsilon=\bar \epsilon_{k_0}$, we have that
\begin{equation}\label{59}
\left\{
\begin{array}{ll}
\bar \epsilon_{k_0}(\eta(\bar A)+\eta'(\bar A)\bar B)\Delta\dot{A}+(\bar\rho-1)\dot{A}+\bar A\dot{\rho}+(\eta(\bar A)+\eta'(\bar A)\bar B)\Delta \bar A_{k_0}=\dot{\nu}(\bar \epsilon_{k_0})\bar A_{k_0},\\
\Delta\dot{\rho}-\frac{2\bar \rho f'(\bar A)}{f(\bar A)}\Delta\dot{A}-\lambda_0\bar\rho\dot{A}-\lambda_0\bar A\dot{\rho}=\dot{\nu}(\bar \epsilon_{k_0})\bar \rho_{k_0},
\end{array}
\right.
\end{equation}
where $\dot A=\frac{\partial A(  \epsilon,x)}{\partial  \epsilon}\vert_{\epsilon=\bar \epsilon_{k_0}}$ and $\dot \rho=\frac{\partial \rho(\epsilon,x)}{\partial \epsilon}\vert_{\epsilon=\bar \epsilon_{k_0}}$.  Testing (\ref{59}) by $\Phi_{k_0}$ gives rise to
\begin{equation}\label{510}
\begin{split}
&\begin{pmatrix}
-\bar \epsilon_{k_0}(\eta(\bar A)+\eta\prime(\bar A)\bar B)\sigma_{k_0}+\bar\rho-1& \bar A  \\
\frac{2\bar\rho f'(\bar A)}{f(\bar A)}\sigma_{k_0}-\lambda_0\bar\rho& -\sigma_{k_0}-\lambda_0\bar A \\
\end{pmatrix}
\begin{pmatrix}
\int_\Omega \dot A \Phi_{k_0}dx \\
\int_\Omega \dot \rho \Phi_{k_0}dx
\end{pmatrix}\\
=&\begin{pmatrix}
\dot \nu(\bar \epsilon_{k_0})Q_{k_0}+\bar \epsilon_{k_0}(\eta(\bar A)+\eta'(\bar A)\bar B)Q_{k_0}\\
\dot\nu(\bar \epsilon_{k_0}),
\end{pmatrix}
\end{split}
\end{equation}
where we have used the fact that $\int_\Omega \Phi^2_{k_0}dx=1$. The coefficient matrix in (\ref{510}) is singular in light of (\ref{47}), therefore
\[\frac{\bar A}{-\sigma_{k_0}-\lambda_0\bar A}=\frac{\dot \nu(\bar \epsilon_{k_0})Q_{k_0}+\bar \epsilon_{k_0}(\eta(\bar A)+\eta'(\bar A)\bar B)Q_{k_0}}{\dot\nu(\bar \epsilon_{k_0})}\]
and consequently we can easily show that $\dot\nu(\bar \epsilon_{k_0})<0$.  According to (1.17) in Theorem 1.16 of \cite{CR2}, $\mu_{k_0}(s)$ and $-s\bar \epsilon'_{k_0}(s)\dot{\nu}(\bar \epsilon_{k_0})$ have the same zeros and the same signs near $s=0$, and for $\mu_{k_0}(s) \neq0$
\[\lim_{s\rightarrow 0}\frac{-s\bar \epsilon'_{k_0}(s)\dot{\nu}(\bar \epsilon_{k_0})}{\mu_{k_0}(s)}=1,\]
where the dot sign denotes the differentiation with respect to $\epsilon$.  Therefore, we have that $\text{sgn}(\mu_{k_0}(s))=\text{sgn}(K_1)$ if $K_1\neq0$ and $\text{sgn}(\mu_{k_0}(s))=\text{sgn}(K_2)$ if $K_1=0$ and $K_2\neq0$.  On the other hand, we already see from the analysis above that the other eigen-value of (\ref{56}) for $k=k_0$ is negative, therefore Theorem \ref{theorem51} readily follows from the arguments above.
\end{proof}
\begin{remark}
Theorem \ref{theorem51} indicates that, only the $k_0$--th bifurcation branch $\Gamma_{k_0}(s)$ around $(\bar A,\bar \rho)$ can be stable.  In other words, if a spatial pattern is stable, it must be on the branch $\Gamma_{k_0}(s)$ for which $\bar \epsilon_k$ is maximized over $\mathbb N^+$.  This selection of principal wavemode provides essential understandings of pattern formations in model (\ref{12}), i.e., stable patterns must develop in terms of the principal mode $\Phi_{k_0}$ if $\epsilon$ is taken to be smaller than but close to $\bar \epsilon_{k_0}$.  An important implication of this wavemode selection mechanism is that larger domain tends to support stable patterns with more modes.  To elucidate this result, we consider the one--dimensional domain $\Omega=(0,L)$ which has $\sigma_k=\big(\frac{k\pi}{L}\big)^2$ as its Neumann eigenvalue.  Denote $\bar \epsilon_{k_1}=\max_{k\in \mathbb N^+} \bar \epsilon_k$ for $L=L_1$ and $\bar \epsilon_{k_2}=\max_{k\in \mathbb N^+} \bar \epsilon_k$ for $L=L_2$, then we must have that $k_1<k_2$ if $L_1<L_2$.  Similar results hold in multi-dimensional domains (under Dirichlet Boundary conditions).  Figure \ref{fig2} in Section \ref{section6} verifies this observation numerically.  Such wavemode selection mechanism was found for a volume filling chemotaxis model with logistic growth in \cite{MOW}.
\end{remark}

\subsection{Bifurcation of transcritical type}
We proceed to find the values of $K_1$ and $K_2$ given in (\ref{51}).  To this end, we first give the following Taylor expansions from straightforward calculations
\begin{equation}\label{511}
\begin{split}
\eta(A)=&\eta(\bar A)+s\eta'(\bar A)Q_k\Phi_k+s^2\Big(\eta'(\bar A)\psi_1+\frac{1}{2}\eta''(\bar A)Q_k^2\Phi_k^2\Big)+s^3\Big(\eta'(\bar A)\psi_2\\
&+\eta''(\bar A)Q_k\psi_1\Phi_k+\frac{\eta'''(\bar A)}{6}Q_k^3\Phi_k^3\Big)+o(s^3).
\end{split}
\end{equation}
Substituting (\ref{51}) and (\ref{511}) into (\ref{41}), we collect the $s^2$--terms there to obtain
\begin{equation}\label{512}
\begin{split}
K_1(\eta(\bar A)+\eta'(\bar A)\bar B)Q_k\sigma_k\Phi_k=&(\eta(\bar A)+\eta'(\bar A)\bar B)\bar\epsilon_k\Delta\psi_1+(\bar\rho-1)\psi_1+\bar A\varphi_1\\
&+(2\eta'(\bar A)+\eta''(\bar A)\bar B)Q_k^2\bar\epsilon_k\vert\nabla\Phi_k\vert^2\\
&+(Q_k-(2\eta'(\bar A)+\eta''(\bar A)\bar B)Q_k^2\bar\epsilon_k\sigma_k)\Phi_k^2.
\end{split}
\end{equation}
\begin{proof}[Proof\nopunct]\emph{of Proposition} \ref{proposition51}.
Multiplying (\ref{512}) by $\Phi_k$ and then integrating it over $\Omega$ by parts, we obtain that
\begin{equation}\label{513}
\begin{split}
&K_1(\eta(\bar A)+\eta'(\bar A)\bar B)Q_k\sigma_k \\
=&\Big(\bar\rho-1-(\eta(\bar A)+\eta'(\bar A)\bar B)\bar\epsilon_k\sigma_k\Big)\int_{\Omega}\psi_1\Phi_kdx+\bar A\int_{\Omega}\varphi_1\Phi_kdx\\
&+(2\eta'(\bar A)+\eta''(\bar A)\bar B)Q_k^2\bar\epsilon_k\int_{\Omega}\Phi_k\vert\nabla\Phi_k\vert^2dx\\
&+(Q_k-(2\eta'(\bar A)+\eta''(\bar A)\bar B)Q_k^2\bar\epsilon_k\sigma_k)\int_{\Omega}\Phi_k^3dx\\
=&\big(\bar\rho-1-(\eta(\bar A)+\eta'(\bar A)\bar B)\bar\epsilon_k\sigma_k\big)\int_{\Omega}\psi_1\Phi_kdx+\bar A\int_{\Omega}\varphi_1\Phi_kdx\\
&+\big(Q_k-(\eta'(\bar A)+\frac{1}{2}\eta''(\bar A)\bar B)Q_k^2\bar\epsilon_k\sigma_k\big)\int_{\Omega}\Phi_k^3dx,
\end{split}
\end{equation}
where we have applied the fact $\int_{\Omega}\vert\nabla\Phi_k\vert^2\Phi_kdx=\frac{\sigma_k}{2}\int_{\Omega}\Phi_k^3dx$ because
\begin{equation*}
\begin{split}
\int_{\Omega}\Phi_k\vert\nabla\Phi_k\vert^2dx&=\frac{1}{2}\int_{\Omega}\nabla\Phi_k^2\cdot\nabla\Phi_kdx=\frac{1}{2}\int_{\Omega}
(\nabla\cdot(\Phi_k^2\nabla\Phi_k)-\Phi_k^2\Delta\Phi_k)dx\\
&=-\frac{1}{2}\int_{\Omega}\Phi_k^2\Delta\Phi_kdx=\frac{\sigma_k}{2}\int_{\Omega}\Phi_k^3dx.
\end{split}
\end{equation*}

Substituting (\ref{51}) into the $\rho$--equation of (\ref{41}), we collect the $s^2$--terms to have
\begin{equation}\label{514}
\begin{split}
\Delta\varphi_1&-\frac{2\bar\rho f'(\bar A)}{f(\bar A)}\Delta\psi_1-\lambda_0\bar\rho\psi_1-\lambda_0\bar A\varphi_1-\Big(\frac{2f'(\bar A)}{f(\bar A)}+2\bar\rho\Big(\frac{f''(\bar A)}{f(\bar A)}-\frac{f'(\bar A)^2}{f(\bar A)^2}\Big)Q_k\Big)Q_k\vert\nabla\Phi_k\vert^2\\
&+\Big(\big(\frac{2f'(\bar A)}{f(\bar A)}+2\bar\rho\Big(\frac{f''(\bar A)}{f(\bar A)}-\frac{f'(\bar A)^2}{f(\bar A)^2}\Big)Q_k\big)Q_k\sigma_k-\lambda_0Q_k\Big)\Phi_k^2=0.
\end{split}
\end{equation}
We test this equation by $\Phi_k$ over $\Omega$ and obtain
\begin{equation}\label{515}
\begin{split}
&\Big(\frac{2\bar\rho f'(\bar A)}{f(\bar A)}\sigma_k-\lambda_0\bar\rho\Big)\int_{\Omega}\psi_1\Phi_kdx-(\sigma_k+\lambda_0\bar A)\int_{\Omega}\varphi_1\Phi_kdx\\
=&\left(\lambda_0Q_k-\Big(\frac{f'(\bar A)}{f(\bar A)}+\bar\rho (\frac{f''(\bar A)}{f(\bar A)}-\frac{f'(\bar A)^2}{f(\bar A)^2})Q_k\Big)Q_k\sigma_k\right)\int_{\Omega}\Phi_k^3dx;
\end{split}
\end{equation}
on the other hand, $(\psi_1,\varphi_1)\in\mathcal{Z}$ in (\ref{412}) implies that
\begin{equation}\label{516}
\int_{\Omega}(Q_k\psi_1+\varphi_1)\Phi_kdx=Q_k\int_{\Omega}\psi_1\Phi_kdx+\int_{\Omega}\varphi_1\Phi_kdx=0
\end{equation}
and solving (\ref{515}) and (\ref{516}) gives us
\begin{equation}\label{517}
\int_{\Omega}\psi_1\Phi_kdx=\frac{\left(\lambda_0Q_k-(\frac{f'(\bar A)}{f(\bar A)}+\bar\rho (\frac{f''(\bar A)}{f(\bar A)}-\frac{f'(\bar A)^2}{f(\bar A)^2})Q_k)Q_k\sigma_k\right)\int_{\Omega}\Phi_k^3dx}{\frac{2\bar\rho f'(\bar A)}{f(\bar A)}\sigma_k-\lambda_0\bar\rho+(\sigma_k+\lambda_0\bar A)Q_k},
\end{equation}
and
\begin{equation}\label{518}
\int_{\Omega}\varphi_1\Phi_kdx=-Q_k\int_{\Omega}\psi_1\Phi_kdx.
\end{equation}
Substituting (\ref{517}) and (\ref{518}) into (\ref{513}), we can easily show (\ref{52}) and this concludes the proof of Proposition \ref{proposition51}.
\end{proof}

We see that $K_1$ is determined by the value of $\int_{\Omega}\Phi_k^3dx$ and system parameters.  By the same calculations, we can also verify $K_1$ given in (\ref{54}) holds for (\ref{42}).

\subsection{Bifurcation of pitchfork type}
If $\Omega$ has a geometry such that $\int_{\Omega}\Phi_k^3dx=0$, for example a finite interval or multi--dimensional rectangle, $\Phi_k$ is a cosine function or a product of cosine functions, which implies that $K_1=0$, therefore we need to find $K_2$ given in (\ref{53}) to determine the stability of the bifurcating solutions.
\begin{proof}[Proof\nopunct]\emph{of Proposition} \ref{proposition52}.
We equate the $s^3$--terms in (\ref{41}) to obtain that
\begin{equation}\label{519}
\begin{split}
&K_2(\eta(\bar A)+\eta'(\bar A)\bar B)Q_k\sigma_k\Phi_k\\
=&(2\eta'(\bar A)+\eta''(\bar A)\bar B)Q_k\bar\epsilon_k(\Delta\psi_1\Phi_k+2\nabla\psi_1\cdot\nabla\Phi_k-\sigma_k\psi_1\Phi_k)\\
&+(\eta(\bar A)+\eta'(\bar A)\bar B)\bar\epsilon_k\Delta\psi_2+(\bar\rho-1)\psi_2+\psi_1\Phi_k+\bar A\varphi_2\\
&+Q_k\varphi_1\Phi_k+\Big(\frac{\eta''(\bar A)}{2}+\frac{\eta'''(\bar A)}{6}\bar B\Big)Q_k^3\bar\epsilon_k(6\Phi_k\vert\nabla\Phi_k\vert^2-3\sigma_k\Phi_k^3).
\end{split}
\end{equation}
Multiply both hand sides of (\ref{519}) by $\Phi_k$ and integrate it over $\Omega$ by parts, then in light of the identity $\int_{\Omega}\Phi_k^2\vert\nabla\Phi_k\vert^2dx=\frac{\sigma_k}{3}\int_{\Omega}\Phi_k^4dx$, we can prove (\ref{53}) in Proposition \ref{proposition52}.
\end{proof}

$K_2$ in (\ref{53}) is determined by the integrals $\int_\Omega \psi_2 \Phi_k,\int_\Omega \varphi_2 \Phi_k, \int_\Omega \psi_1 \Phi_k^2$ and $\int_\Omega \varphi_1 \Phi_k^2$, as well as the system parameters.  For the sake of completeness, we proceed to evaluate these integrals.  Equating the $s^3$--terms of the second equation in (\ref{41}), we see that
\begin{equation}\label{520}
\begin{split}
&\Delta\varphi_2-\frac{2\bar\rho f'(\bar A)}{f(\bar A)}\Delta\psi_2-\lambda_0\bar A\varphi_2-\lambda_0\bar\rho\psi_2-\Big(\frac{2f'(\bar A)}{f(\bar A)}+2\bar\rho(\frac{f''(\bar A)}{f(\bar A)}-\frac{f'(\bar A)^2}{f(\bar A)^2})Q_k\Big)\Delta\psi_1\Phi_k\\
&-\Big(\frac{2f'(\bar A)}{f(\bar A)}+4\bar\rho (\frac{f''(\bar A)}{f(\bar A)}-\frac{f'(\bar A)^2}{f(\bar A)^2})Q_k\Big)\nabla\psi_1\nabla\Phi_k -\frac{2f'(\bar A)}{f(\bar A)}Q_k\nabla\varphi_1\nabla\Phi_k \\
&+\Big(2\bar\rho (\frac{f''(\bar A)}{f(\bar A)}-\frac{f'(\bar A)^2}{f(\bar A)^2})Q_k\sigma_k-\lambda_0\Big)\psi_1\Phi_k+\Big(\frac{2f'(\bar A)}{f(\bar A)}Q_k\sigma_k-\lambda_0Q_k\Big)\varphi_1\Phi_k\\
&-\Big(2(\frac{f''(\bar A)}{f(\bar A)}-\frac{f'(\bar A)^2}{f(\bar A)^2})+\bar\rho (\frac{f'''(\bar A)}{f(\bar A)}-\frac{3f'(\bar A)f''(\bar A)}{f(\bar A)^2}+\frac{2f'(\bar A)^3}{f(\bar A)^3})Q_k\Big)Q_k^2(2\Phi_k\vert\nabla\Phi_k\vert^2-\sigma_k\Phi_k^3)=0. \nonumber
\end{split}
\end{equation}
Similar as above, we test (\ref{520}) by $\Phi_k$ and have that
\begin{equation}\label{521}
\begin{split}
&\Big(\frac{2\bar\rho f'(\bar A)}{f(\bar A)}\sigma_k-\lambda_0\bar\rho\Big)\int_{\Omega}\psi_2\Phi_kdx-(\sigma_k+\lambda_0\bar A)\int_{\Omega}\varphi_2\Phi_kdx\\
=&\Big(\lambda_0-\frac{2f'(\bar A)}{f(\bar A)}\sigma_k-2\bar\rho (\frac{f''(\bar A)}{f(\bar A)}-\frac{f'(\bar A)^2}{f(\bar A)^2})Q_k\sigma_k\Big)\int_{\Omega}\psi_1\Phi_k^2dx+\lambda_0Q_k\int_{\Omega}\varphi_1\Phi_k^2dx\\
&+\frac{2f'(\bar A)}{f(\bar A)}\int_{\Omega}\psi_1\vert\nabla\Phi_k\vert^2dx-\frac{2f'(\bar A)}{f(\bar A)}Q_k\int_{\Omega}\varphi_1\vert\nabla\Phi_k\vert^2dx\\
&-\frac{1}{3}\Big(\bar\rho (\frac{f'''(\bar A)}{f(\bar A)}-\frac{3f'(\bar A)f''(\bar A)}{f(\bar A)^2}+\frac{2f'(\bar A)^3}{f(\bar A)^3})Q_k+2(\frac{f''(\bar A)}{f(\bar A)}-\frac{f'(\bar A)^2}{f(\bar A)^2})\Big)Q_k^2\sigma_k\int_{\Omega}\Phi_k^4dx,
\end{split}
\end{equation}
where we have applied the following identities which can be obtained through straightforward calculations,
 \[\int_{\Omega}\Delta\psi_1\Phi_k^2dx=2\int_{\Omega}\psi_1\vert\nabla\Phi_k\vert^2dx-2\sigma_k\int_{\Omega}\psi_1\Phi_k^2dx\]
and
\[\int_{\Omega}\nabla\psi_1\cdot\Phi_k\nabla\Phi_k dx=\sigma_k\int_{\Omega}\psi_1\Phi_k^2dx-\int_{\Omega}\psi_1\vert\nabla\Phi_k\vert^2dx.\]
On the other hand, since $(\psi_2,\varphi_2)$ satisfies (\ref{412}), we can evaluate $\int_\Omega \psi_2 \Phi_k $ and $\int_\Omega \varphi_2 \Phi_k $ in terms of $\int_\Omega \psi_1 \vert \nabla \Phi_k\vert^2 $ and $\int_\Omega \varphi_1 \vert \nabla \Phi_k\vert^2 $ in the rest part.

Multiplying (\ref{512}) and (\ref{514}) by $\vert\nabla\Phi_k\vert^2$ and then integrating them over $\Omega$ by parts, thanks to $K_1=0$ we have that
\begin{equation}\label{522}
\begin{split}
&(\eta(\bar A)+\eta'(\bar A)\bar B)\bar\epsilon_k\int_{\Omega}\Delta\psi_1\vert\nabla\Phi_k\vert^2dx+(\bar\rho-1)\int_{\Omega}\psi_1\vert\nabla\Phi_k\vert^2dx\\
+&\bar A\int_{\Omega}\varphi_1\vert\nabla\Phi_k\vert^2dx+(Q_k-(2\eta'(\bar A)+\eta''(\bar A)\bar B)Q_k^2\bar\epsilon_k\sigma_k)\int_{\Omega}\Phi_k^2\vert\nabla\Phi_k\vert^2dx\\
+&(2\eta'(\bar A)+\eta''(\bar A)\bar B)Q_k^2\bar\epsilon_k\int_{\Omega}\vert\nabla\Phi_k\vert^4dx=0\\
\end{split}
\end{equation}
and
\begin{equation}\label{523}
\begin{split}
&\int_{\Omega}\Delta\varphi_1\vert\nabla\Phi_k\vert^2dx-\frac{2\bar\rho f'(\bar A)}{f(\bar A)}\int_{\Omega}\Delta\psi_1\vert\nabla\Phi_k\vert^2dx\\
-&\lambda_0\bar\rho\int_{\Omega}\psi_1\vert\nabla\Phi_k\vert^2dx-\lambda_0\bar A\int_{\Omega}\varphi_1
\vert\nabla\Phi_k\vert^2dx-\lambda_0Q_k\int_{\Omega}\Phi_k^2\vert\nabla\Phi_k\vert^2dx\\
-&\Big(\frac{2f'(\bar A)}{f(\bar A)}Q_k+2\bar\rho (\frac{f''(\bar A)}{f(\bar A)}-\frac{f'(\bar A)^2}{f(\bar A)^2})Q_k^2\Big)\int_{\Omega}\vert\nabla\Phi_k\vert^2(\vert\nabla\Phi_k\vert^2-\sigma_k\Phi_k^2)dx=0,
\end{split}
\end{equation}
respectively, where we have applied the following fact
\begin{equation*}
\begin{split}
\int_{\Omega}\Delta\psi_1\vert\nabla\Phi_k\vert^2dx&=\int_{\Omega}(\nabla\cdot(\nabla\psi_1\vert\nabla\Phi_k\vert^2)-\nabla\psi_1\cdot\nabla\vert\nabla\Phi_k\vert^2)dx=-\int_{\Omega}\nabla\psi_1\cdot\nabla\vert\nabla\Phi_k\vert^2dx\\
&=2\sigma_k\int_{\Omega}\nabla\psi_1\cdot\Phi_k\nabla\Phi_kdx=\sigma_k\int_{\Omega}\nabla\psi_1\cdot\nabla\Phi_k^2dx\\
&=2\sigma_k^2\int_{\Omega}\psi_1\Phi_k^2dx-2\sigma_k\int_{\Omega}\psi_1\vert\nabla\Phi_k\vert^2dx.
\end{split}
\end{equation*}

To simplify (\ref{522}) and (\ref{523}), we want to apply the following identity
\begin{equation}\label{524}
\int_{\Omega}\vert\nabla\Phi_k\vert^4dx=\sigma_k^2\int_{\Omega} \Phi_k^4dx.
\end{equation}
which follows from straightforward calculations.  To see that (\ref{524}) holds, we multiply $2\vert\nabla\Phi_k\vert^4=\vert\nabla\Phi_k\vert^2\Delta\Phi_k^2+2\sigma_k\Phi_k^2\vert\nabla\Phi_k\vert^2$ by $\vert\nabla\Phi_k\vert^2$ and then integrate it over $\Omega$ by parts to have that
\begin{equation}\label{525}
\int_{\Omega}\vert\nabla\Phi_k\vert^4dx=\frac{1}{2}\int_{\Omega}\vert\nabla\Phi_k\vert^2\Delta\Phi_k^2dx+\sigma_k\int_{\Omega}\Phi_k^2\vert\nabla\Phi_k\vert^2dx.
\end{equation}
On the other hand, the Green's identities imply that
\begin{equation}\label{526}
\int_{\Omega}\vert\nabla\Phi_k\vert^2\Delta\Phi_k^2dx=2\sigma_k^2\int_{\Omega}\Phi_k^4dx-2\sigma_k\int_{\Omega}\Phi_k^2\vert\nabla\Phi_k\vert^2dx,
\end{equation}
therefore (\ref{524}) is an immediate consequence of (\ref{525}) and (\ref{526}).  Now we conclude from (\ref{522}) and (\ref{523}) that
\begin{equation}\label{527}
\begin{split}
&(\bar\rho-1-2(\eta(\bar A)+\eta'(\bar A)\bar B)\bar\epsilon_k\sigma_k)\int_{\Omega}\psi_1\vert\nabla\Phi_k\vert^2dx\\
&+\bar A\int_{\Omega}\varphi_1\vert\nabla\Phi_k\vert^2dx+2(\eta(\bar A)+\eta'(\bar A)\bar B)\bar\epsilon_k\sigma_k^2\int_{\Omega}\psi_1\Phi_k^2dx\\
=&-\Big(\frac{2}{3}(2\eta'(\bar A)+\eta''(\bar A)\bar B)Q_k^2\bar\epsilon_k\sigma_k^2+\frac{Q_k\sigma_k}{3}\Big)\int_{\Omega}\Phi_k^4dx
\end{split}
\end{equation}
and
\begin{equation}\label{528}
\begin{split}
&(\lambda_0\bar\rho-\frac{4\bar\rho f'(\bar A)}{f(\bar A)}\sigma_k)\int_{\Omega}\psi_1\vert\nabla\Phi_k\vert^2dx+(\lambda_0\bar A+2\sigma_k)\int_{\Omega}\varphi_1\vert\nabla\Phi_k\vert^2dx\\
&+\frac{4\bar\rho f'(\bar A)}{f(\bar A)}\sigma_k^2\int_{\Omega}\psi_1\Phi_k^2dx-2\sigma_k^2\int_{\Omega}\varphi_1\Phi_k^2dx\\
=&-\Big(\frac{2}{3}(\frac{2f'(\bar A)}{f(\bar A)}Q_k+2\bar\rho \Big(\frac{f''(\bar A)}{f(\bar A)}-\frac{f'(\bar A)^2}{f(\bar A)^2}\Big)Q_k^2)\sigma_k^2+\frac{\lambda_0Q_k\sigma_k}{3}\Big)\int_{\Omega}\Phi_k^4dx.
\end{split}
\end{equation}
Multiply (\ref{512}) with $K_1=0$ and (\ref{514}) by $\Phi_k^2$ and integrate them over $\Omega$ by parts.  Together with (\ref{527}) and (\ref{528}), we have that
\begin{equation}\label{529}
\begin{split}
&\left(
\begin{array}{cccc}
\bar\rho-1-2(\eta(\bar A)+\eta'(\bar A)\bar B)\bar\epsilon_k\sigma_k & \bar A & 2(\eta(\bar A)+\eta'(\bar A)\bar B)\bar\epsilon_k\sigma_k^2 & 0\\
\lambda_0\bar\rho-\frac{4\bar\rho f'(\bar A)}{f(\bar A)}\sigma_k & \lambda_0\bar A+2\sigma_k & \frac{4\bar\rho f'(\bar A)}{f(\bar A)}\sigma_k^2 & -2\sigma_k^2\\
2(\eta(\bar A)+\eta'(\bar A)\bar B)\bar\epsilon_k & 0 & \bar\rho-1-2(\eta(\bar A)+\eta'(\bar A)\bar B)\bar\epsilon_k\sigma_k & \bar A\\
\frac{4\bar\rho f'(\bar A)}{f(\bar A)} & -2 & \lambda_0\bar\rho-\frac{4\bar\rho f'(\bar A)}{f(\bar A)}\sigma_k & \lambda_0\bar A+2\sigma_k
\end{array}
\right)\\
\cdot&\left(
\begin{array}{cccc}
\int_{\Omega}\psi_1\vert\nabla\Phi_k\vert^2dx\\
\int_{\Omega}\varphi_1\vert\nabla\Phi_k\vert^2dx\\
\int_{\Omega}\psi_1\Phi_k^2dx\\
\int_{\Omega}\varphi_1\Phi_k^2dx
\end{array}
\right)=\left(
\begin{array}{cccc}
-(\frac{2}{3}(2 \eta'(\bar A)+\eta''(\bar A)\bar B)Q_k^2\bar\epsilon_k\sigma_k^2+\frac{Q_k \sigma_k}{3})\\
-(\frac{2}{3}(\frac{2f'(\bar A)}{f(\bar A)}Q_k+2\bar\rho (\frac{f''(\bar A)}{f(\bar A)}-\frac{f'(\bar A)^2}{f(\bar A)^2})Q_k^2)\sigma_k^2+\frac{\lambda_0Q_k\sigma_k}{3})\\
(\frac{2}{3}(2 \eta'(\bar A)+\eta''(\bar A)\bar B)Q_k^2\bar\epsilon_k\sigma_k-Q_k)\\
(\frac{2}{3}(\frac{2f'(\bar A)}{f(\bar A)}Q_k+2\bar\rho (\frac{f''(\bar A)}{f(\bar A)}-\frac{f'(\bar A)^2}{f(\bar A)^2})Q_k^2)\sigma_k-\lambda_0Q_k)
\end{array}
\right) \int_{\Omega}\Phi_k^4dx .
\end{split}
\end{equation}
Through straightforward calculations, we have that
\begin{equation*}
\begin{split}
&\begin{vmatrix}
\begin{array}{cccc}
\bar\rho-1-2(\eta(\bar A)+\eta'(\bar A)\bar B)\bar\epsilon_k\sigma_k & \bar A & 2(\eta(\bar A)+\eta'(\bar A)\bar B)\bar\epsilon_k\sigma_k^2 & 0\\
\lambda_0\bar\rho-\frac{4\bar\rho f'(\bar A)}{f(\bar A)}\sigma_k & \lambda_0\bar A+2\sigma_k & \frac{4\bar\rho f'(\bar A)}{f(\bar A)}\sigma_k^2 & -2\sigma_k^2\\
2(\eta(\bar A)+\eta'(\bar A)\bar B)\bar\epsilon_k & 0 & \bar\rho-1-2(\eta(\bar A)+\eta'(\bar A)\bar B)\bar\epsilon_k\sigma_k & \bar A\\
\frac{4\bar\rho f'(\bar A)}{f(\bar A)} & -2 & \lambda_0\bar\rho-\frac{4\bar\rho f'(\bar A)}{f(\bar A)}\sigma_k & \lambda_0\bar A+2\sigma_k
\end{array}
\end{vmatrix}\\
=&-\lambda_0\bar A\begin{vmatrix}
\begin{array}{cc}
\bar\rho-1-4(\eta(\bar A)+\eta'(\bar A)\bar B)\bar\epsilon_k\sigma_k & \bar A\\
\lambda_0\bar\rho-\frac{8\bar\rho f'(\bar A)}{f(\bar A)}\sigma_k & \lambda_0\bar A+4\sigma_k
\end{array}
\end{vmatrix}\\
=&\frac{3\lambda_0\bar A}{\sigma_k+\lambda_0\bar A}\left(4\Big(\frac{2\bar B}{f(\bar A)}f'(\bar A)+\bar\rho-1\Big)\sigma_k^2-5\lambda_0\bar A\sigma_k-\lambda_0^2\bar A^2\right),
\end{split}
\end{equation*}
then (\ref{529}) is solvable provided that
\[\frac{2\bar B}{f(\bar A)}f'(\bar A)+\bar\rho-1\neq \frac{\lambda_0\bar A(\lambda_0\bar A+5\bar \sigma_k)}{4\sigma_k^2}.\]
Finally by solving system (\ref{529}) we can evaluate all integrals in $K_2$ given by (\ref{53}).  Since the calculations are straightforward but extremely lengthy, we shall skip the details here.  We want to remark that, if $\sigma_{2k}=4\sigma_{k}$, this condition above is embedded by the necessary condition in Theorem \ref{theorem42} that $\epsilon_k\neq \epsilon_{2k}$.  This condition is always satisfied in the case when $\Omega=(0,L)$ for which $\sigma_k=(k\pi/L)^2$.

\section{Numerical simulations}\label{section6}
In this section, we perform extensive numerical simulations of models (\ref{12}) and (\ref{13}) over one--dimensional interval $\Omega=(0,L)$ and two--dimensional square $\Omega=(0,L)\times (0,L)$ to illustrate our theoretical results and to demonstrate the self--organized spatial temporal dynamics of the systems.  First of all, the Neumann Laplacian eigen-pair over the interval $\Omega=(0,L)$ is
\[\Big(\frac{k\pi}{L} \Big)^2\leftrightarrow \cos \frac{k\pi x}{L}, k\in \mathbb N^+.\]
We refer $\cos \frac{k\pi x}{L}$ as the wavemode and $k$ as the wavemode number for $\Omega=(0,L)$.  Similarly, $\cos \frac{m\pi x}{L} \cos \frac{n\pi y}{L}$ is a wavemode and $(m,n)$ is a wavemode pair for $\Omega=(0,L) \times (0,L)$.  Through numerical simulations, we are concerned with the effect of the diffusion rate $\epsilon$ and the domain size $L$ on our wavemode selection mechanism as well as the spatial-temporal behaviors, in particular the formation of stable aggregates.  Throughout the rest of this section, we choose the nonlinear diffusion and sensitivity functions to be $\eta(A)=1-e^{-A}$ and $f(A)=\log(A+1)$ respectively.  We shall call the stable aggregate in 1D a spike and that in 2D a hotspot in the sequel.

\subsection{1D numerics}

First of all, we explore models (\ref{12}) and (\ref{13}) over interval $\Omega=(0,L)$.  Theorem \ref{theorem41} and Theorem \ref{theorem42} state that the only stable wavemode of (\ref{12}) and (\ref{13}) is $\cos \frac{k_0\pi x}{L}$, where $k_0$ is a positive integer that maximizes $\bar \epsilon_k$ in (\ref{47}) and (\ref{414}) with $\sigma_k=(\frac{k\pi}{L})^2$, respectively.
\begin{table}[h!]
\centering
\begin{tabular}{ |p{1cm}|p{1cm}|p{1cm}|p{1cm}|p{1cm}|p{1cm}|p{1cm}|p{1cm}|}
\hline
$k$ & 1 &2& 3& 4& 5& 6 & 7 \\
\hline
$\bar \epsilon_k$ & 0.0335 &   0.0091 &   0.0041  &  0.0023  &  0.0015  &  0.0010 &   0.0008 \\
 \hline
\end{tabular}
\caption{List of bifurcation values $\bar \epsilon_k$ of system (\ref{12}) given in (\ref{47}).  $\bar \epsilon_1=\max_{k\in \mathbb N^+} \bar \epsilon_k$ and the stable wavemode is $\cos \pi x$.}
\label{tab1}
\end{table}
Taking $A^0=1$, $\bar B=2$, $\lambda_0=0.1$ and $L=1$ in model (\ref{12}), we present a list of bifurcation values $\bar \epsilon_k$ of (\ref{47}) in Table \ref{tab1}.  It is easy to see that $\max \bar \epsilon_k$ is achieved at $\bar \epsilon_1=0.0335$.  According to Theorem \ref{theorem51}, the wavemode that drives the instability of $(\bar A,\bar \rho)$ to model (\ref{12}) must be $\cos \pi x$, which is spatially monotone decreasing; wavemode $\cos k \pi x$ is always unstable for all $k\geq1$, therefore stable and monotone patterns must develop in the form of $\cos \pi x$.  We verify this selection mechanism in Figure \ref{fig1} numerically, where spatial-temporal solutions are plotted to illustrate the formations of stable steady states with boundary layer to system (\ref{12}) over $\Omega=(0,1)$, with time $t=0$, 100 and 500.  $\epsilon$ is selected to be 0.029, which is around the first bifurcation value $\bar \epsilon_1=0.0035$; the initial conditions are the small perturbations $(A_0(x),\rho_0(x))=(\bar A,\bar \rho)+(0.01,0.01)\cos 4\pi x$, which have a wavemode $\cos 4\pi x$. We see that $A(x,t)$ and $\rho(x,t)$ evolve according to the stable monotone mode $\cos \pi x$.
\begin{figure}[h!]
\centering
  \includegraphics[width=4.8in]{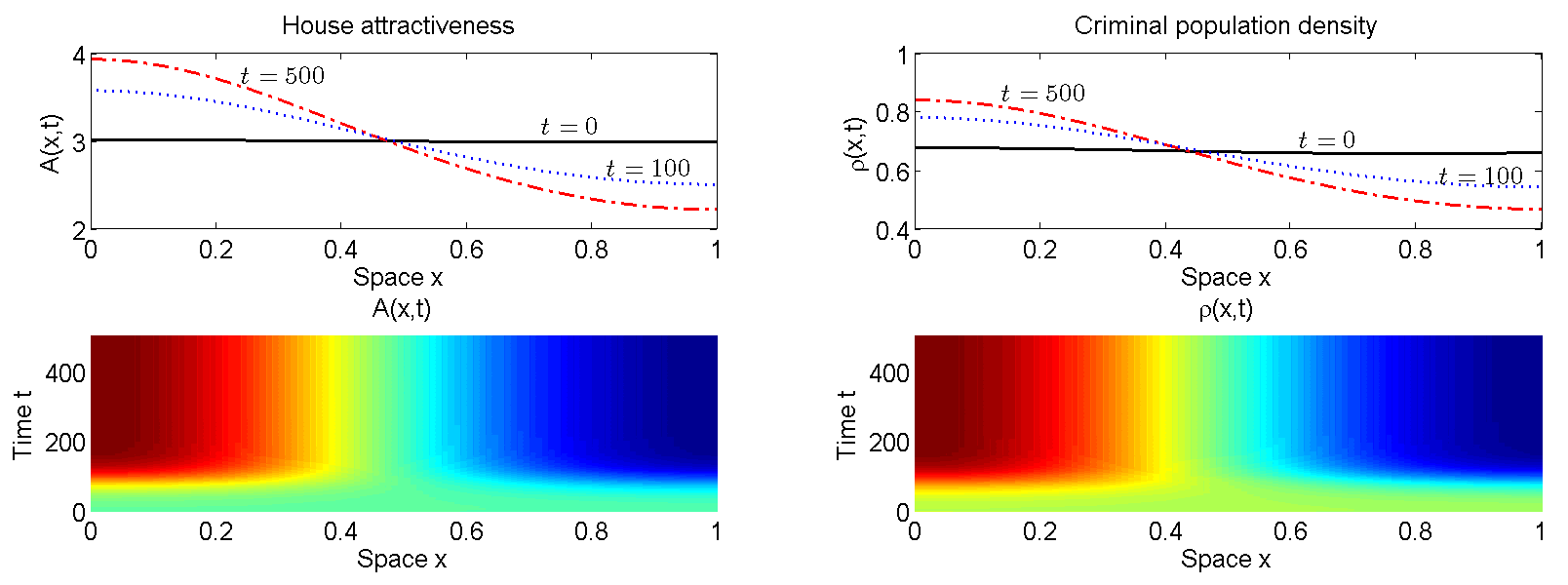}
 \caption{Emergence of stable monotone steady states to (\ref{12}).  Diffusion rate $\epsilon=0.029$ is close to the first bifurcation value $\bar \epsilon_1=0.0335$.}\label{fig1}
\end{figure}

\subsubsection{Effect of interval length on wavemode selection mechanism}
We next study the change of spatial profiles of stable steady states with respect to the variation of interval length.   Without loss of our generality, we only consider model (\ref{12}) here and the same results can be obtained for (\ref{13}).  First of all, we observe that if the interval length $L$ is sufficiently small, bifurcation value in (\ref{47}) satisfies
\[\bar \epsilon_k= \frac{\big(\frac{2\bar Bf'(\bar A)}{f(\bar A)}+\bar\rho-1\big)(\frac{k \pi}{L})^2-\lambda_0\bar A}{(\eta(\bar A)+\eta'(\bar A)\bar B)((\frac{k \pi}{L})^2+\lambda_0\bar A)(\frac{k \pi}{L})^2} \approx   \frac{\frac{2\bar Bf'(\bar A)}{f(\bar A)}+\bar\rho-1}{\eta(\bar A)+\eta'(\bar A)\bar B}   \Big(\frac{L}{k\pi}\Big)^2,k\in \mathbb N^+,\]
there $\max_{k\in \mathbb N^+}\bar \epsilon_k=\bar \epsilon_1$ and only the first wavemode $\cos \frac{\pi x}{L}$ can be stable, which is spatially monotone.  This implies that small intervals only support monotone stable patterns and large intervals support nonmonotone stable patterns.  Indeed, since both $A$ and $\rho$ satisfy Neumann boundary condition, we can construct nonmonotone solutions by reflecting and periodically extending the monotone solution at the boundary.  For example, if $(A(x,t),\rho(x,t))$ is a monotone solution to (\ref{12}) over $(0,1)$, then $(A(2-x,t),\rho(2-x,t))$ is a solution over $(1,2)$, therefore we have a solution over $(0,2)$ which is nonmonotone spatially.  On the other hand, since $\sigma_k=(\frac{k\pi}{L})^2$, wave number $k$ and interval length $L$ must be proportional when $\bar \epsilon_k$ achieves its maximum.  In general terms, if $\cos \frac{k_0 \pi x}{L}$ is the wavemode for $(0,L)$, then the wavemode for $(0,2L)$ must be $\cos \frac{2k_0 \pi x}{2L}$, which has a wave number $2k_0$.  All these observations are consistent with our wavemode selection mechanism.  The same assertions can be made about $\bar \epsilon_k$ in (\ref{413}) and system (\ref{13}).

In Table \ref{tab2}, we list the wave numbers and maximum bifurcation values for different values of $L$.  Then wavemode of (\ref{12}) and (\ref{13}) takes the form $\cos \frac{k_0 \pi x}{L}$ and we expect stable patterns to develop in this form.  Numerical simulations are performed in Figure \ref{fig2} to verify our analysis.
\begin{table}[h!]
  \centering
  \begin{tabular}{c|c|c|c|c|c|cc}
  \hline
   Domain size $L$ & &3& 5& 7& 9 & 11    \\
    \hline
    \multirow{2}{*}{Model (\ref{12})} & $k_0$ & 1 & 2 & 3 & 4 & 5   \\
    \hhline{~-------}
    & $\bar \epsilon^{(\ref{12})}_{k_0}$ & 0.0781 & 0.1004 & 0.1003  &0.0991 & 0.0981 \\
    \hline
    \multirow{2}{*}{Model (\ref{13})} & $k_0$ & 1 & 2 & 3 & 4 & 5 \\
    \hhline{~-------}
    & $\bar \epsilon^{(\ref{13})}_{k_0}$ & 0.0963& 0.1239 & 0.1238 &0.1224 &0.1221   \\
    \hline
  \end{tabular}
   \begin{tabular}{c|c|c|c|c|c|cc}
  \hline
   Domain size $L$ &  & 13& 15& 17& 19 &21  \\
    \hline
    \multirow{2}{*}{Model (\ref{12})} & $k_0$ & 5  & 6 &7 &8 &9 \\
    \hhline{~-------}
    & $\bar \epsilon^{(\ref{12})}_{k_0}$ & 0.0988&0.1004&0.1008&0.1006 &0.1001\\
    \hline
    \multirow{2}{*}{Model (\ref{13})} & $k_0$ & 5 & 6 &7 &8&9\\
    \hhline{~-------}
    & $\bar \epsilon^{(\ref{13})}_{k_0}$ &0.1219 &0.1239 &0.1243 & 0.1242 &0.1238 \\
    \hline
  \end{tabular}
\caption{Wavemode number that maximizes the bifurcation value for different intervals.  We see that system (\ref{41}) and (\ref{42}) have the same wavemode number, which increases as domain size increases. The first bifurcation value $\max_{k\in \mathbb N^+} \bar \epsilon_k$ of system (\ref{12}) is smaller than that of system (\ref{13}).}
\label{tab2}
\end{table}

\begin{figure}[h!]
\centering
  \includegraphics[width=4.8in]{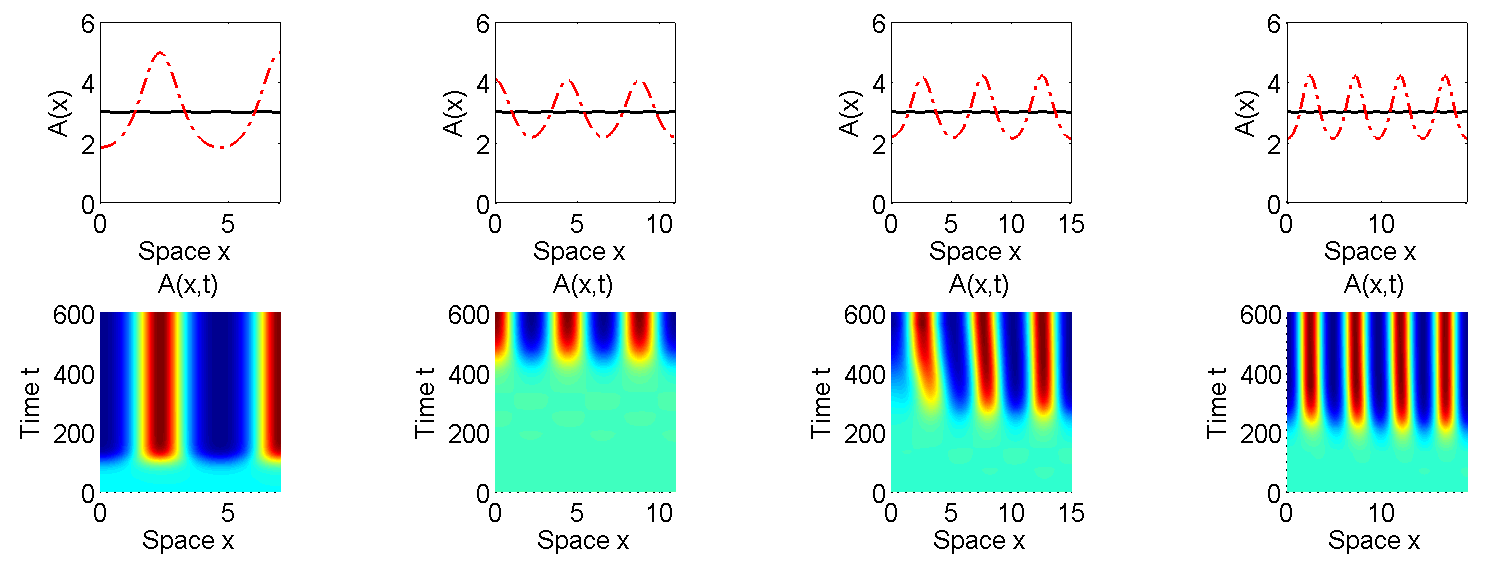}
 \caption{Emergence of stable multi-spike steady states to model (\ref{12}) with domain size $L=7$, 11, 15 and 19 from the left to the right.  Initial condition is plotted in solid curve and steady state in dash-dotted lines.  Parameters are $A^0=1$, $\bar B=2$, $\lambda=0.1$ and $\epsilon=0.09$.  According to our theoretical results and Table \ref{tab2}, the wavemodes are $\cos\frac{3\pi x}{7}$, $\cos\frac{ 5\pi x}{11 }$, $\cos\frac{6 \pi x}{15 }$ and $\cos\frac{8 \pi x}{19 }$, respectively.  Plots in this figure verify this exact wavemode selection mechanism.}\label{fig2}
\end{figure}

\subsubsection{Effect of small diffusion rate $\epsilon$}
In Figure \ref{fig3}, we plot the steady state of (\ref{12}) over $(0,1)$ for $\epsilon>0$ being sufficiently smaller than the maximum bifurcation value $\epsilon_{k_0}$, where all the rest parameters and initial data are taken to be the same as those in Figure \ref{fig1}.  Our numerical simulations show that both $A(x)$ and $\rho(x)$ are monotone decreasing for $\epsilon$ being small.  Moreover as $\epsilon$ shrinks to zero, steady state $A(x)$ approaches to a boundary spike in the form of a $\delta$-function and $\rho(x)$ to a boundary spike with bounded maximum value.  This result indicates that if the nearby victimization effect is extremely weak, one may expect the emergence of the clusterings of crime data.  However, rigorous analysis of the spiky solution is a quite delicate problem and it is out of the scope of our paper.  See \cite{KWW} for mathematical analysis on spiky solutions of (\ref{11}) for instance.
\begin{figure}[h!]
\centering
  \includegraphics[width=3.8in]{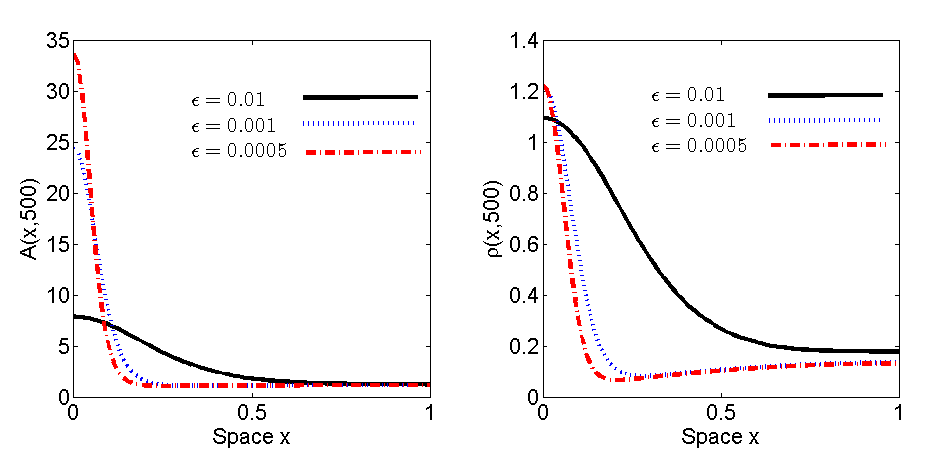}
 \caption{Plots of stable boundary spikes for small diffusion rate $\epsilon$.  Our numerical solutions suggest that as $\epsilon$ shrinks to zero $A(x)$ converges to a $\delta$-type boundary spike and $\rho(x)$ converges to a bounded boundary spike.}\label{fig3}
\end{figure}
\begin{figure}[h!]
\centering
  \includegraphics[width=4.8in]{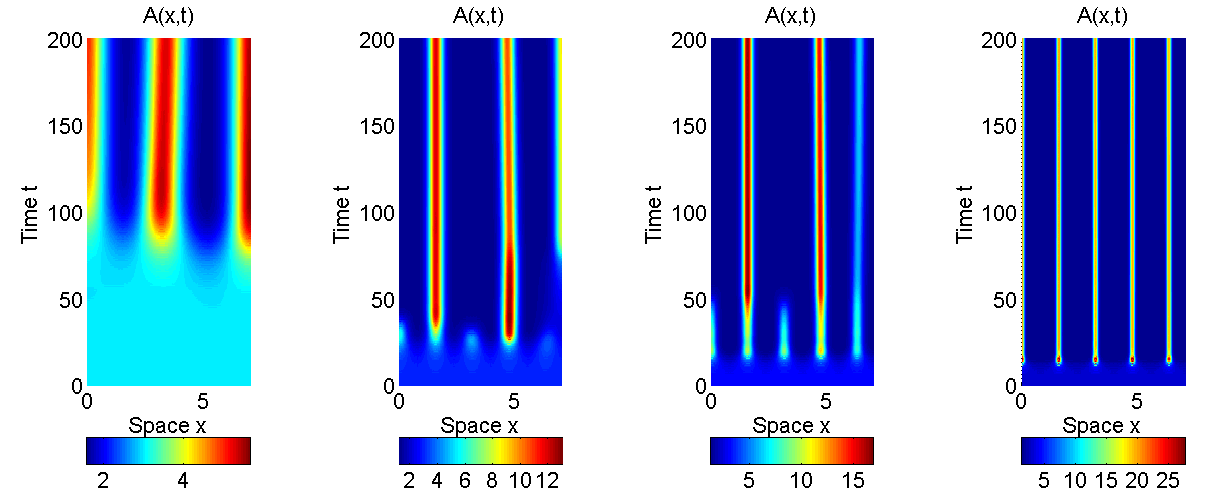}
 \caption{Formation of stable spikes of model (\ref{12}) for different values of $\epsilon$, which is taken to be 0.05, 0.01, 0.005 and 0.001 in each plot from the left to the right respectively.  Number of aggregates increases as $\epsilon$ shrinks to zero.}\label{fig4}
\end{figure}

\subsubsection{Difference between model (\ref{12}) and model (\ref{13})}

Next we compare the pattern formations in the departure--dependent system (\ref{12}) and the arrival--dependent system (\ref{13}).  We want to remind that both systems admit $(\bar A,\bar \rho)$ as their homogeneous steady state solutions.   To elucidate their differences, we choose the same parameters and initial data for both systems as those for  Figure \ref{fig2}.  We remind that, according to Table \ref{tab2}, both models have the same wavemode number $k_0$ which increases as domain size $L$ increases.  On the other hand, the first bifurcation value $\max_{k\in \mathbb N^+} \bar \epsilon_k$ of (\ref{12}) is always smaller than that of system (\ref{13}), therefore $(\bar A, \bar \rho)$ loses its stability in (\ref{12}) for smaller value of $\epsilon$ than (\ref{13}).  This suggests that, in light of simulations in Figure \ref{fig3}, (\ref{13}) can develop spikes which have larger amplitudes than (\ref{12}) for each fixed small $\epsilon$.  This is numerically illustrated in Figure \ref{fig5}, where we select $\epsilon$ to be far away from the $\bar \epsilon_{k_0}$.
 \begin{figure}[h!]
\centering
  \includegraphics[width=4.8in]{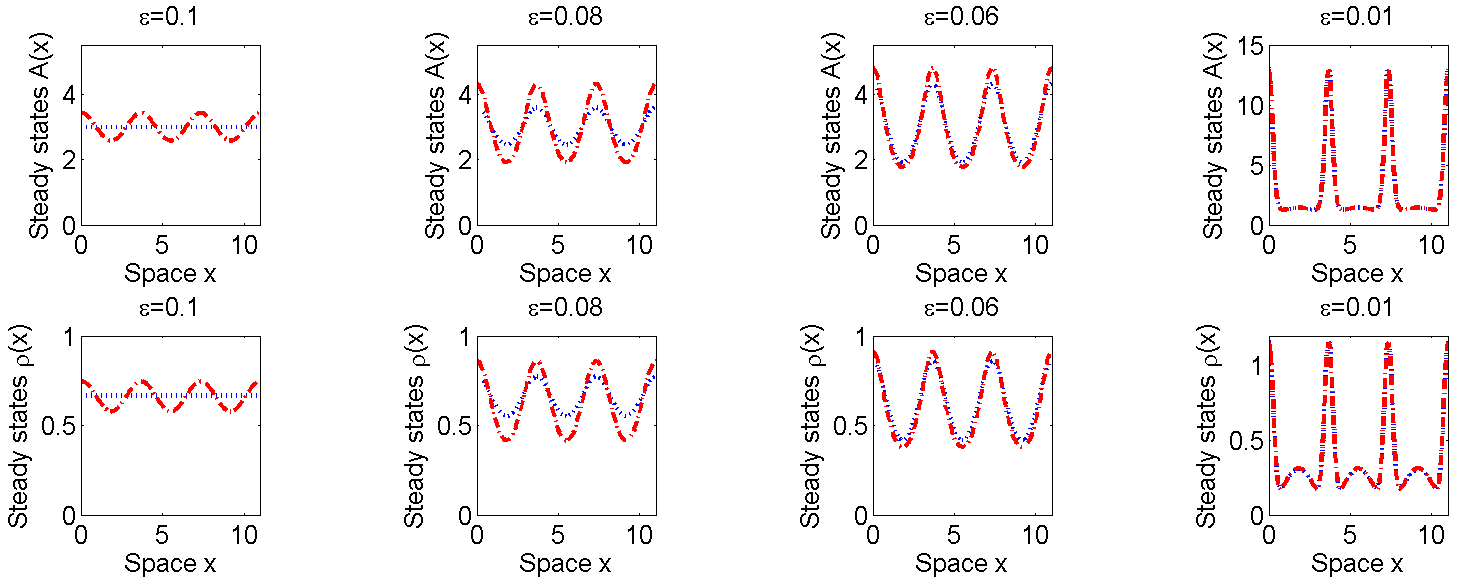}
 \caption{Stable steady states of systems (\ref{12}) and (\ref{13}) for different values of $\epsilon$.  Steady states of the departure--dependent model (\ref{12}) are given in dotted blue lines and steady states of the arrival--dependent model (\ref{13}) are given in dash--dotted red lines.  The amplitude of spikes of model (\ref{12}) is always larger than those of model (\ref{13}).   This suggests that the arrival--dependent model (\ref{13}) is more sensitive to the variation of near repeat victimization effect $\epsilon$ than the departure--dependent model (\ref{12}).  This is consistent with the linear stability results in Section \ref{section3}.  For $\epsilon$ being sufficiently small, both models have the same structures in their stable steady states; moreover, the stable patterns no longer fall into the category of our wavemode selection mechanism. }\label{fig5}
\end{figure}
In Figure \ref{fig5}, we numerically solve system (\ref{12}) and system (\ref{13}) over $(0,10)$ for different values of $\epsilon$.  Parameters are chosen to be $A^0=1$, $\bar B=3$ and the initial data are $A_0(x)=\bar A+0.01\cos \frac{\pi x}{4}$ and $\rho_0(x)=\bar \rho+0.01\cos 2\pi x$.  There are several conclusions that we can draw out of Figure \ref{fig5}.  First of all, the numerical simulations there verify that system (\ref{12}) and (\ref{13}) have the same wavemode section mechanism which is consistent with our theoretical analysis above.  See Table \ref{tab2} and the discussions there for example.  Moreover, the amplitude of patterns to (\ref{12}) is larger than that of (\ref{13}), while both systems have the same stable steady states for $\epsilon$ being sufficiently small.  We can also find that small $\epsilon$ tends to support the formation of stable steady states with spikes and amplitude of each spike increases as $\epsilon$ goes to zero.

\subsubsection{Merging and emerging of interior spikes}

In Figure \ref{fig6}, we observe the merging of spikes of house attractiveness for model (\ref{12}) in plot (a) and the emerging of new spikes for model (\ref{13} in plot (b)).  These patterning processes are referred to as coarsening processes which have been observed in chemotaxis model with logistic growth in \cite{PH}.  In plot (a), we choose $\epsilon=0.026$, $A^0=1$, $\bar B=9.11$ and $\lambda=0.1$.  Initial data are $(A_0(x),\rho_0(x)=(\bar A,\bar \rho)+(0.001,0.001) \cos(1.2 \pi x)$.  We observe that two spikes merge to form a single spike at around $x=8$ and $x=13$.  In plot (b), we choose $\epsilon=0.03$, $A_0=0.967$, $\bar B=13.09$ and $\lambda=0.1$.  Initial data are $(A_0(x),\rho_0(x)=(\bar A,\bar \rho)+(0.0007,0.0007) \cos(0.3 \pi x)$.  We observe that two spikes emerge from a single spike at around $x=8$ and $x=17.5$.  Shifting of spikes is observed in both figures.  Our numerical simulations also suggest that there is usually no coarsening when the domain size is small, which often occurs when domain size and $\bar A$ are large.
\begin{figure}[h!]
\centering
  \includegraphics[width=4.8in]{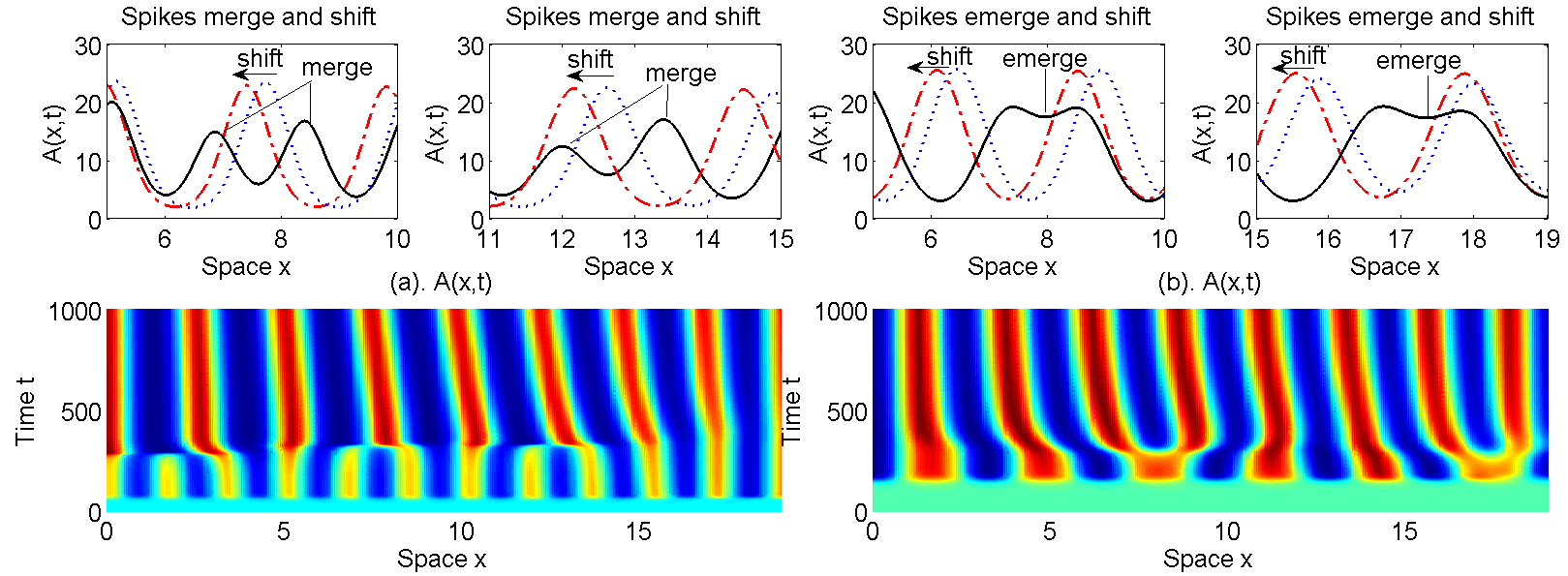}
 \caption{Merging and emerging of spikes.  In Figure (a) and Figure (b), we plot the merging and emerging of spikes to system (\ref{12}) and system (\ref{13}) respectively.  }\label{fig6}
\end{figure}

\subsection{2D numerics}

We now study models (\ref{12}) and (\ref{13}) over the two-dimensional domain $\Omega=(0,L)\times (0,L)$. In this special case, the Neumann eigen-pairs are
\[\Big(\frac{m \pi}{L}\Big)^2+\Big(\frac{m \pi}{L}\Big)^2\leftrightarrow \cos \frac{m\pi x}{L}\cos \frac{n\pi y}{L}, (m,n) \in \mathbb N^+ \times \mathbb N^+ \cup \{(1,0),(0,1)\}.\]
According to Theorem \ref{theorem41} and Theorem \ref{theorem42}, each eigenvalue above gives rise to a bifurcation value for models (\ref{12}) and (\ref{13}) when the Laplace eigen--value is simple.  However out simulations here indicates that these results hold even in the case when $\Omega$ is a square for which the eigen--value is not simple.  We also want to point out that the boundary of the square is not smooth at the corners, therefore the Neumann boundary condition and the bifurcation solutions can be interpreted as the weak solutions through the first Green’s identity.  One can show that the weak solutions are classical except at the corners by elliptic embeddings in the standard way.

\begin{table}[h!]
\centering
\begin{tabular}{|l|l|l|l|l|l|l|}
\hline
\multicolumn{7}{|c|}{Bifurcation values $\bar \epsilon_{mn}$}\\ \hline
    ~ & $m=0$ & $m=1$ & $m=2$ & $m=3$ & $m=4$ & $m=5$ \\ \hline
    $n=0$ & undefined &  0.0094   & 0.0091  &  0.0047 &   0.0028  &  0.0019 \\ \hline
    $n=1$ & 0.0094   & 0.0129  &  0.0077    &0.0043   & 0.0027&    0.0018 \\ \hline
    $n=2$ &0.0091  &  0.0077   & 0.0053 &   0.0034  &  0.0023  &  0.0016 \\ \hline
    $n=3$ & 0.0047 &   0.0043  &  0.0034   &0.0025   & 0.0019  &  0.0014 \\ \hline
    $n=4$ & 0.0028   & 0.0027   & 0.0023  &  0.0019  &  0.0015  &  0.0012 \\ \hline
    $n=5$ & 0.0019   & 0.0018 &   0.0016  &  0.0014  &  0.0012  &  0.0009\\ \hline
\end{tabular}
\caption{List of bifurcation values $\bar \epsilon_{mn}$ in (\ref{47}) for system (\ref{12}) over $\Omega=(0,1)\times(0,1)$.  Parameters are taken to be $A^0=1$, $\bar B=3$ and $\lambda=0.9$.  We see that $\bar \epsilon_{11}=\max \bar \epsilon_{mn}$ and the stable wavemode is $\cos \pi x \cos \pi y$.}
\label{tab3}
\end{table}

\subsubsection{Effect of domain size on wavemode selection mechanism}

In Figure \ref{fig7}, we plot the numerical solutions of system (\ref{12}) over the 2D domain $\Omega=(0,1)\times (0,1)$ to illustrate our result in Table \ref{tab3}.  The parameters are chosen to be $A^0=1$, $\bar B=3$ and $\lambda=0.9$.  The initial conditions are taken to be small perturbations from the homogeneous steady state $(\bar A, \bar \rho)$.
\begin{figure}[h!]
\centering
  \includegraphics[width=4.8in]{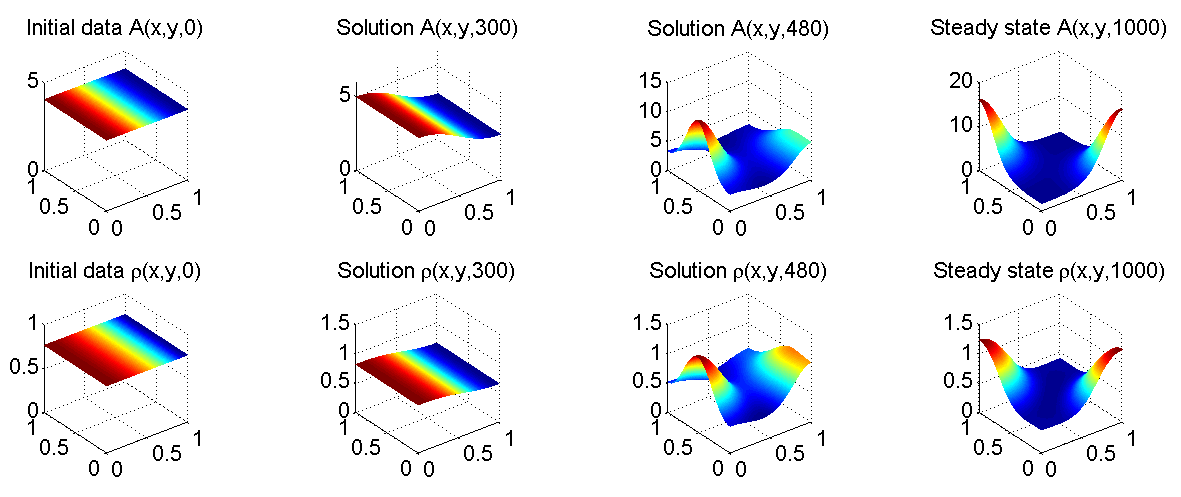}
 \caption{Formation of double-spike steady state.  We choose the parameters to be $\epsilon=0.008$, $A^0=1$, $\bar B=3$ and $\lambda=0.9$.  Stable generating wavemode is $\cos \pi x \cos \pi y$.  The formation of stable steady state corresponds to the wavemode selection as in our analysis.}\label{fig7}
\end{figure}

According to our theoretical results, the stable wavemode of (\ref{12}) over $\Omega=(0,L)\times (0,L)$  must be $\cos \frac{m_0 \pi x}{L} \cos \frac{n_0 \pi y}{L}$ for which $(m_0,n_0)\neq (0,0)$ maximizes $\bar \epsilon_{m_0n_0}$ in (\ref{47}), with $\sigma_k$ being replaced by $\big(\frac{m \pi}{L}\big)^2+\big(\frac{m \pi}{L}\big)^2$.  Models (\ref{12}) and (\ref{13}) admit nontrivial patterns that develop in form of this mode if $\epsilon$ is around $\bar \epsilon_{m_0n_0}$.  Similar as in 1D, by periodically reflecting and extending hotspot at the boundary of square $(0,L)\times (0,L)$, we can construct steady states with more spikes to larger domains $(0,mL)\times (0,nL)$, $(m.n)\in \mathbb N^+\times \mathbb N^+$.  This observation suggests that large domain supports more hotspots than small domain.  Actually, we can verify this by our wavemode section mechanism.  Table \ref{tab4} gives the wavemode series for a list of square domains with different side lengths.
\begin{table}[h!]
  \centering
  \begin{tabular}{p{1.8cm}|m{1.2cm}|m{0.85cm}|m{0.85cm}|m{0.85cm}|m{0.85cm}|m{0.85cm}|m{0.85cm}|m{0.85cm}|m{0.85cm} m{0.85cm} m{0.85cm}}
  \hline
   Domain size $L$ & & 2 &  3  & 4&  5 &6&7& 8&9  \\
    \hline
    \multirow{2}{*}{Model (\ref{12})} &$(m_0,n_0)$ &(2,2) & (4,0), (0,4) &(1,5), (5,1) & (4,5), (5,4) & (3,7), (7,3)& (4,8), (8,4)& (2,10), (10,2)&(7,9), (9,7) \\
    \hhline{~-----------}
    & $\bar \epsilon^{(\ref{12})}_{m_0,n_0}$ & 0.0129 & 0.0132 & 0.0133  &0.0133 & 0.0133& 0.0133&0.0133&0.0133 \\
    \hline
    \multirow{2}{*}{Model (\ref{13})} & $(m_0,n_0)$ &(2,2) & (4,0), (0,4) &(1,5), (5,1) & (4,5), (5,4) & (3,7), (7,3)& (4,8), (8,4)& (2,10), (10,2)&(7,9), (9,7) \\
    \hhline{~-----------}
    & $\bar \epsilon^{(\ref{13})}_{(m_0,n_0)}$ & 0.0145& 0.0148 & 0.0149 &0.0149 &0.0149 &0.0149 &0.0149 &0.0149  \\
    \hline
  \end{tabular}
\caption{Wavemode number that maximizes the bifurcation value $\bar \epsilon_{mn}$ of systems (\ref{12}) and (\ref{13}) for different values of domain size $L$.  The parameters are taken to be $A^0=1$, $\bar B=3$ and $\lambda=0.9$.  We see that model (\ref{12}) and (\ref{13}) have the same wavemode selection mechanism.  Moreover, the increase in domain size decreases the maximum bifurcation values for both models.  Complex patterns with aggregates usually emerge for large domains.}
\label{tab4}
\end{table}

\begin{figure}[h!]
\centering
  \includegraphics[width=4.8in]{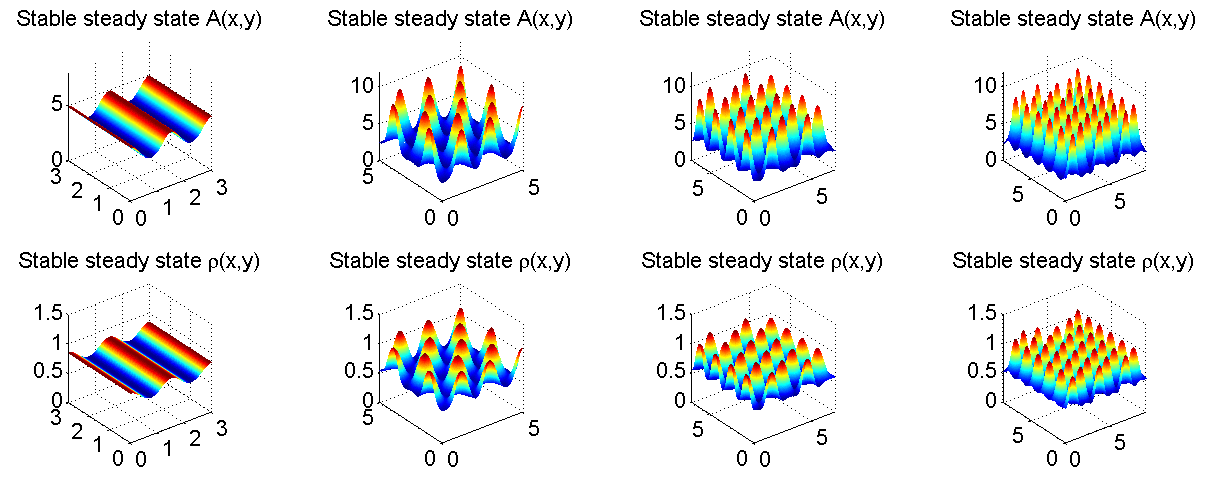}
 \caption{Stable hotspots to house attractiveness and criminal density of model (\ref{12}) for different values of domain size.  We choose the parameters to be $\epsilon=0.00127$, $A^0=1$, $\bar B=3$ and $\lambda=0.9$.  The initial conditions are taken to be $A_0(x,y)=\bar A+0.01\cos \pi x$ and $\rho_0(x,y)=\bar \rho+0.01\cos \frac{\pi x}{4} $.  According to our theoretical results and Table \ref{tab4}, the wavemodes are $\cos \frac{4\pi x }{3}$, $\cos \frac{ 4\pi x}{5}\cos \frac{ 5\pi y}{5 }$, $\cos \frac{ 4\pi x}{ 7}\cos \frac{ 8\pi y}{7 }$ and $\cos \frac{ 7\pi x}{ 9}\cos \frac{ 9\pi y}{ 9}$ respectively.  Stable steady states correspond to these wavemodes and verify selection mechanism in 2D.}\label{fig8}
\end{figure}

\subsubsection{Effect of small diffusion rate $\epsilon$}

Finally, we study in Figure \ref{fig9} the qualitative behaviors of stationary hotspots to model (\ref{12}) for $\epsilon$ being greatly smaller than its maximum bifurcation point $\bar \epsilon_{m_0n_0}$.  The parameters are taken to be the same as those in Figure \ref{fig8}.  Similar as in the 1D case, our numerical simulations suggest that small diffusion rate $\epsilon$ increases the number of stable hotspots and the magnitude of the aggregates.
\begin{figure}[h!]
\centering
  \includegraphics[width=4.8in]{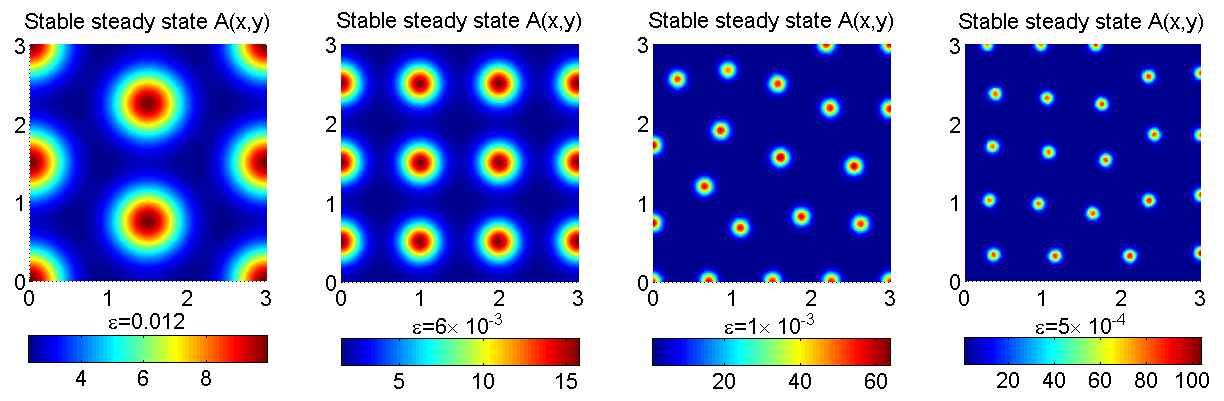}
 \caption{Stable hotspots to house attractiveness as diffusion rate $\epsilon$ shrinks to zero.  We surmise that $A(x,y)$ converges to a linear combination of $\delta$-functions as $\epsilon$ goes to zero.  Stable hotspots are also observed in the criminal population density.}\label{fig9}
\end{figure}

\section{Conclusions and Discussions}\label{section7}

In this paper, we study the stationary solutions to the general $2\times 2$ reaction--diffusion models (\ref{12}) and (\ref{13}) over multi-dimensional bounded domains subject to homogeneous Neumann boundary conditions.  These systems are generalizations of the urban crime model proposed by Short \emph{et al}.  Our modification is made based on the assumption that both the near--repeat victimization effect and perception of attractiveness are heterogeneous, i.e., dependent on the attractiveness of its current and/or neighbouring sites.

First of all, we carry out the linear stability analysis of the homogenous equilibrium $(\bar A,\bar \rho)$.  It is shown that this trivial solution loses its stability when the intrinsic near--repeat victimization effect $\epsilon$ becomes large.  Then we proceed to investigate the existence of nonhomogeneous steady state that bifurcates from the homogeneous steady state in the manner of Crandall--Rabinowitz local bifurcation theory.  Moreover, we perform the stability analysis of these nontrivial patterns by detailed calculations.  Compared to the bifurcation analysis in \cite{CCM}, our results treat a wider class of urban crime models; moreover, we find the exact formula of $K_2$ for the pitch-fork bifurcation when $K_1=0$, which was not done in \cite{CCM}.

Our stability results provide a selection mechanism of principal wavemode in the formations of nontrivial patterns.  To be precise, the only stable wavemode must be a Neumann eigenfunction $\Phi_{k_0}$, whose wavemode number $k_0$ is a positive integer that maximizes the bifurcation value $\bar \epsilon_k$ over positive integers; moreover, the pattern $(A_k(s,x),\rho_k(s,x))$ is unstable for all $k\neq k_0$.  Based on the stable wavemode selection mechanism, we can precisely predict the formation of stable patterns of the house attractiveness and criminal population density that have large amplitude such as boundary spikes, interior spikes in 1D and hotspots and hotstripes in 2D.

Numerical simulations have been performed to verify our theoretical results when the domain is a one--dimensional interval or a two--dimensional square.  There are also several important findings in our numerics.  First of all, they indicate that both small diffusion rate $\epsilon$ and large domain size support the emergence of stable aggregates.  Moreover, the amplitude of house attractiveness $A(x)$ and criminal population density $\rho(x)$ increase as $\epsilon$ decreases.  In particular, $A(x)$ approaches to a $\delta$-type function as $\epsilon$ shrinks to zero.  These results suggest that small nearby victimization effect tends to support the clustering of criminal data, and urban regions have much more complicated criminal behaviors than small rural areas.  We also observe that coarsening process ocuurs where two spikes merging into one new spike or one spike breaking into two separate spikes, in 1D simulations.

There are also some other unsolved questions concerning the stationary solutions of models (\ref{12}) and (\ref{13}).  Theorem \ref{theorem41}, \ref{theorem42} and Theorem \ref{theorem51} give us the existence and stability of their nontrivial positive steady states, which are small perturbation from the homogeneous equilibrium.  It is interesting to investigate the emergence of large amplitude solutions such as the spikes and hotspots demonstrated in our numerical simulations.   Another interesting but delicate problem is to investigate the stability of these aggregates.  Our bifurcation analysis is performed based on the local theory of \cite{CR} and it is interesting and important to study its global continuums when $\epsilon$ is away from the bifurcation point $\bar \epsilon_k$.  According to Remark \ref{remark41}, for each bifurcation branch, one has to determine whether its continuum is noncompact, i.e., approaches to infinity, or intersects the $\epsilon$--axis at $\epsilon^*$, which should not one of the bifurcation values.  However, this can be a quite challenging problem even over one--dimensional domains.

When $\epsilon$ is chosen to be close to the principal bifurcation value $\max_{k\in \mathbb N^+} \bar \epsilon_k$, our wavemode selection mechanism provides a very useful and effective way to predict the formation of spikes and hotspots in systems (\ref{12}) and (\ref{13}).  However, when $\epsilon$ is far away from the principal bifurcation value, i.e., being sufficiently small, rigorous mathematical analysis of the dynamics of the aggregates is needed to fully understand the pattern formation in these models.  For example, the mechanism that drives the merging and emerging of spikes in Figure \ref{fig6} is a delicate question that deserves future explorations.  From the viewpoint of mathematical analysis, it is also important to investigate dynamics of the time-dependent systems (\ref{12}) and (\ref{13}), including but not limited to questions such as global existence, convergence to the steady states, traveling wave solutions, etc.

\medskip
\medskip

\end{document}